\documentclass[11pt,a4paper,reqno]{amsart}
\usepackage[english]{babel}
\usepackage[T1]{fontenc}
\usepackage{verbatim}
\usepackage{palatino}
\usepackage{amsmath}
\usepackage{mathabx}
\usepackage{amssymb}
\usepackage{amsthm}
\usepackage{amsfonts}
\usepackage{graphicx}
\usepackage{esint}
\usepackage{color}
\usepackage{mathtools}

\usepackage[colorlinks = true, citecolor = black]{hyperref}
\title[Loomis-Whitney inequalities in $\He^n$]{Loomis-Whitney inequalities in Heisenberg groups}
\author{Katrin F\"assler and Andrea Pinamonti}
\address{Department of Mathematics and Statistics \\ University of Jyv\"askyl\"a, P.O. Box. 35 (MaD), FI-40014 University of Jyv\"askyl\"a \\ Finland}
\address{Department of Mathematics\\
University of Trento,
Via Sommarive 14\\
I-38123 Povo\\
Italy} \email{katrin.s.fassler@jyu.fi}
\email{andrea.pinamonti@unitn.it}
\date{\today}
\thanks{K.F. is supported by the Academy of Finland via the project \emph{Singular integrals, harmonic functions, and boundary regularity in
Heisenberg groups}, grant Nos. 321696, 328846. A.P. is partially
supported by supported by the University of Trento and GNAMPA of
INDAM}

\newcommand{\R}{\mathbb{R}}
\newcommand{\W}{\mathbb{W}}
\newcommand{\He}{\mathbb{H}}
\newcommand{\N}{\mathbb{N}}

\newcommand{\Z}{\mathbb{Z}}

\newcommand{\calH}{\mathcal{H}}

\newcommand{\1}{\mathbf{1}}

\def\Barint_#1{\mathchoice
          {\mathop{\vrule width 6pt height 3 pt depth -2.5pt
                  \kern -8pt \intop}\nolimits_{#1}}%
          {\mathop{\vrule width 5pt height 3 pt depth -2.6pt
                  \kern -6pt \intop}\nolimits_{#1}}%
          {\mathop{\vrule width 5pt height 3 pt depth -2.6pt
                  \kern -6pt \intop}\nolimits_{#1}}%
          {\mathop{\vrule width 5pt height 3 pt depth -2.6pt
                  \kern -6pt \intop}\nolimits_{#1}}}

\numberwithin{equation}{section}

\theoremstyle{plain}
\newtheorem{thm}[equation]{Theorem}

\newtheorem{lemma}[equation]{Lemma}

\newtheorem{ex}[equation]{Example}

\theoremstyle{definition}

\newtheorem{definition}[equation]{Definition}

\theoremstyle{remark}
\newtheorem{remark}[equation]{Remark}

\newcommand{\nref}[1]{(\hyperref[#1]{#1})}

\begin{document}
\begin{abstract} This note concerns \emph{Loomis-Whitney inequalities} in Heisenberg groups $\He^n$:
\begin{displaymath} |K| \lesssim \prod_{j=1}^{2n}|\pi_j(K)|^{\frac{n+1}{n(2n+1)}}, \qquad K \subset \He^n. \end{displaymath}
Here $\pi_{j}$, $j=1,\ldots,2n$, are the \emph{vertical Heisenberg
projections} to the hyperplanes $\{x_j=0\}$, respectively, and
$|\cdot|$ refers to a natural Haar measure on either $\He^n$, or
one of the hyperplanes. The Loomis-Whitney inequality in the first
Heisenberg group $\mathbb{H}^1$ is a direct consequence of known
$L^p$ improving properties of the standard Radon transform in
$\mathbb{R}^2$. In this note, we show how the Loomis-Whitney
inequalities in higher dimensional Heisenberg groups can be
deduced by an elementary inductive argument from the inequality in
$\mathbb{H}^1$. The same approach, combined with multilinear
interpolation, also yields the following strong type bound:
\begin{displaymath}
\int_{\mathbb{H}^n} \prod_{j=1}^{2n} f_j(\pi_j(p))\;dp\lesssim
\prod_{j=1}^{2n} \|f_j\|_{\frac{n(2n+1)}{n+1}}
\end{displaymath}
for all nonnegative measurable functions $f_1,\ldots,f_{2n}$ on
$\mathbb{R}^{2n}$. These inequalities and their geometric
corollaries are thus ultimately based on planar geometry. Among
the applications of Loomis-Whitney inequalities in $\mathbb{H}^n$,
we mention the following sharper version of the classical
geometric Sobolev inequality in $\mathbb{H}^n$:
\begin{displaymath}
\|u\|_{\frac{2n+2}{2n+1}} \lesssim
\prod_{j=1}^{2n}\|X_ju\|^{\frac{1}{2n}}, \qquad u \in BV(\He^n),
\end{displaymath} where $X_j$, $j=1,\ldots,2n$, are the standard
horizontal vector fields in $\He^n$. Finally, we also establish an
extension of the Loomis-Whitney inequality in $\mathbb{H}^n$,
where the Heisenberg vertical coordinate projections
$\pi_1,\ldots,\pi_{2n}$ are replaced by more general families of
mappings that allow us to apply the same inductive approach based
on the $L^{3/2}$-$L^3$ boundedness of an operator in the plane.
 \end{abstract}

\maketitle


\section{Introduction}

The \emph{Loomis-Whitney} inequality in $\R^{d}$ bounds the volume
of a set $K \subset \R^{d}$ by the areas of its coordinate
projections:
\begin{equation}\label{LWIneq} |K| \leq \prod_{j = 1}^{d} |\tilde \pi_{j}(K)|^{\frac{1}{d - 1}}, \end{equation}
where $\tilde \pi_{j}(x_{1},\ldots,x_{d}) = (x_{1},\ldots,x_{j -
1},x_{j + 1},\ldots,x_{d})$. Here $|A|$ refers to $k$-dimensional
Lebesgue outer measure in $\R^{k}$ whenever $A \subset \R^{k}$.
The inequality \eqref{LWIneq} is due to Loomis and Whitney
\cite{MR0031538} from 1949. It is trivial for $d=2$ and follows by
induction, using H\"older's inequalities, for $d>2$. The
Loomis--Whitney inequality is one of the fundamental inequalities
in geometry and has been studied intensively; we refer to
\cite{Benn, Bob, Salani, Finner,  MR3300318} and references
therein for a historical account and some applications of the
Loomis-Whitney inequality.

The present note discusses analogues of \eqref{LWIneq} in
Heisenberg groups $\mathbb{H}^n$. It arose as a complement to
manuscript \cite{fassler2020planar} with Tuomas Orponen, in which
we reduced the proof of the Loomis-Whitney inequality for
$\mathbb{H}^1$ to an incidence geometric problem in the plane that
we resolved using the method of polynomial partitioning. Later we
learned that the Loomis-Whitney inequality in the first Heisenberg
group -- and inequalities of similar type -- had already been
obtained earlier \cite{MR667786,MR1945281,MR2576685,MR4176542} by
a Fourier-analytic approach or the so-called method of
refinements, albeit not phrased in terms of Heisenberg
projections. In addition to acknowledging previous work, the aim
of the present note is to show how the Loomis-Whitney inequality
in $\mathbb{H}^n$ for $n>1$ can be proven by induction, similarly
as the original inequality \cite{MR0031538}, but now using the
version in $\mathbb{H}^1$ as a base case. Alternatively, one could
apply the method of refinements also for $n>1$, see the related
comment in \cite[Section 4]{MR2895344}. The inductive approach in
the present note has the advantage of easily yielding certain
strong-type endpoint inequalities, see Theorem
\ref{mainIntroStrong}, which are not covered by \cite{MR2895344}
or other literature we are aware of. For applications to geometric
Sobolev and isoperimetric inequalities in $\mathbb{H}^n$, the
weak-type inequalities would however be sufficient.

\textbf{Acknowledgements.} This paper would not have been written
without our previous project with Tuomas Orponen on the subject of
the Loomis-Whitney inequality in $\mathbb{H}^1$. Parts of the
introduction and Section \ref{s:Appl} draw heavily from
\cite{fassler2020planar}. We thank Tuomas for the past
collaboration as well as for valuable suggestions that helped to
improve the exposition of the present paper.

\subsection{Heisenberg groups}\label{ss:Heis}
The \emph{$n$-th Heisenberg group} $\mathbb{H}^n$ is the group
$(\mathbb{R}^{2n+1},\cdot)$ with
\begin{equation}\label{eq:GroupProd} (x,t) \cdot (x',t') := \Big(x + x',  t + t' +
\tfrac{1}{2}\sum_{j=1}^n x_j x_{n+j}'-x_{n+j}x_j'\Big),
\end{equation}
which makes it a nilpotent Lie group of step $2$. Here, $(x,t)$
denotes a point in $\mathbb{R}^{2n+1}$ with
$x=(x_1,\ldots,x_{2n})\in \mathbb{R}^{2n}$ and $t\in \mathbb{R}$.
For $x\in \mathbb{R}^{2n}$ and $k\in \{1,\ldots,2n\}$, we will use
the symbol $\hat x_k$ to denote either the point in
$\mathbb{R}^{2n}$ that is obtained by replacing the $k$-th
coordinate of $x$ with $0$, or the point in $\mathbb{R}^{2n-1}$
that is obtained by simply deleting the $k$-th coordinate of $x$.
The meaning should always be clear from the context.

In geometric measure theory of the sub-Riemannian Heisenberg group
\cite{MR3587666}, an important role is played by \emph{Heisenberg
projections}
 that are adapted to the group and dilation structure of $\mathbb{H}^n$
 and that map onto homogeneous subgroups of $\mathbb{H}^n$.
We only consider  projections associated to the "coordinate"
 hyperplanes containing the $t$-axis, so we limit our discussion
 to those. Let $\W_{j}\subset \He^n$, $j=1,\ldots, 2n$, be the
\emph{($1$-codimensional) vertical subgroups} of $\He^n$ given by
the hyperplanes $\{(x,t)\in\mathbb{R}^{2n+1}:\,x_j=0\}$,
respectively. Write $$\mathbb{L}_{j} :=
\{(0,\ldots,0,x_j,0,\ldots,0) : x_j \in \R\}$$ for the span of the
$j$-th standard basis vector. So $\mathbb{L}_j$ is a
\emph{complementary ($1$-dimensional) horizontal subgroup} of
$\W_{j}$. This means, for example, that every point $p \in \He^n$
has a unique decomposition $p = w_{j} \cdot l_{j}$, where $w_{j}
\in \W_{j}$ and $l_{j} \in \mathbb{L}_{j}$.  These decompositions
give rise to the \emph{vertical coordinate projections}
\begin{displaymath} p \mapsto w_{j} =: \pi_{j}(p) \in \W_{j},\quad j=1,\ldots,2n. \end{displaymath}
Using the group product in \eqref{eq:GroupProd}, it is  easy to
write down explicit expressions for $\pi_{j}$:
\begin{equation}\label{eq:ProjForm}
\pi_{j}(x,t) = (\hat x_j,t + \tfrac{x_j x_{n+j}}{2}) \quad
\text{and} \quad \pi_{n+j}(x,t) = (\hat x_{n+j},t - \tfrac{x_j
x_{n+j}}{2}),\quad j=1,\ldots,n.
\end{equation}
Readers who are not comfortable with the Heisenberg group can
simply identify  $\W_{j}$ with $\R^{2n}$, and consider the maps
\begin{displaymath}
(x,t)\mapsto (x_1,\ldots,x_{j-1},x_{j+1},\ldots,x_{2n},t+
\tfrac{x_j x_{n+j}}{2}),\quad\text{for }j=1,\ldots,n,
\end{displaymath}
and their analogs for $j=n+1,\ldots,2n$,
 without paying attention to their origin. It is clear that the
 projections
$\pi_{1},\ldots,\pi_{2n}$ are smooth, and hence locally Lipschitz
with respect to the Euclidean metric in $\R^{2n+1}$, and they
satisfy
\begin{equation}\label{eq:RankCondition}
\det\left( D \pi_j(p) D \pi_j(p)^t\right)\geq 1,\quad
j=1,\ldots,2n,\;p\in \mathbb{R}^{2n+1}.
\end{equation}
 Vertical projections are, in fact, \textbf{not} Lipschitz
with respect to the \emph{Kor\'anyi distance} $d(p,q) = \|q^{-1}
\cdot p\|$ on $\mathbb{H}^n$. Nonetheless they  play a significant
role in the geometric measure theory of Heisenberg groups --
 as do orthogonal projections in $\R^{d}$ --
 so they have been actively investigated in recent years,
 see \cite{MR3047423,MR2955184,MR3992573,MR3495435,2018arXiv181112559H,2020arXiv200204789H}.
 The vertical projections are non-linear maps,
 but their \emph{fibres} $\pi_{j}^{-1}\{w\}$ are nevertheless lines.
 In fact, the fibres of $\pi_{j}$ are precisely the left translates of the line $\mathbb{L}_j$,
 that is, $\pi_{j}^{-1}\{w\} = w \cdot \mathbb{L}_j$ for $w \in \W_j$.

For subsets of $\He^n \cong \R^{2n+1}$, the notation $|\cdot|$
will refer to Lebesgue (outer) measure on $\R^{2n+1}$, and for
subsets of a vertical plane $\R^{2n} \cong \W_j \subset \He^n$,
the notation $|\cdot|$ will refer to Lebesgue (outer) measure in
$\R^{2n}$. Up to multiplicative constants, they could also be
defined as the $(2n+2)$- and $(2n+1)$-dimensional Hausdorff
measures, respectively, relative to the Kor\'anyi metric on
$\He^n$. So, our measures coincide with canonical "intrinsic"
objects in $\He^n$. All integrations on $\He^n$ or $\W_j$ will be
performed with respect to Lebesgue measures.

\subsection{Loomis-Whitney inequalities in $\He^n$ and their generalizations}
We can now state a variant of the Loomis-Whitney inequality
\eqref{LWIneq} for subsets of $\He^n$ in terms of the vertical
coordinate projections $\pi_{j}$. In $\R^{d}$, the inequality
makes a reference to the $d$ orthogonal coordinate projections
$\widetilde{\pi}_1,\ldots,\widetilde{\pi}_d$. These are, now, best
viewed as the projections whose fibres are translates of lines
parallel to the coordinate axes. In $\He^n$, we consider instead
the vertical projections $\pi_{j}$ whose fibres are left
translates of $\mathbb{L}_{j}$, $j=1,\ldots,2n$; the precise
formulae were stated in \eqref{eq:ProjForm}.
 With this notation, the following
variant of the Loomis-Whitney inequality holds:
\begin{thm}[Loomis-Whitney inequality in $\mathbb{H}^n$]\label{mainIntro2}
Fix $n\in \mathbb{N}$. Let $K \subset \R^{2n+1}$ (or $K \subset
\He^n$) be an arbitrary set. Then
\begin{equation}\label{form18} |K| \lesssim \prod_{j=1}^{2n} |\pi_{j}(K)|^{\frac{n+1}{n(2n+1)}}. \end{equation}
\end{thm}
Here and in the following, the symbol $\lesssim$ indicates that
the inequality holds up to a positive and finite multiplicative
constant on the right-hand side. We only have to prove the
inequality for Lebesgue measurable sets $K\subset
\mathbb{R}^{2n+1}$. In the general case, we simply pick
$G_{\delta}$-sets $K_j\subset \mathbb{R}^{2n}$ with $K_j\supseteq
\pi_j(K)$ and $|K_j|= |\pi_j(K)|$ for $j=1,\ldots, 2n$, assuming
that the right-hand side of \eqref{form18}  is finite. Then
$K':=\bigcap_{j=1}^{2n} \pi_j^{-1}(K_j)$ is a Lebesgue measurable
subset of $\mathbb{R}^{2n+1}$ that contains $K$ and it suffices to
apply the Loomis-Whitney inequality to $K'$.

So we consider only Lebesgue measurable sets $K$ in the following.
By the inner regularity of the Lebesgue measure, Theorem
\ref{mainIntro2} is then equivalent to the validity of
 \eqref{form18} for all \emph{compact} sets $K\subset
\mathbb{R}^{2n+1}$. Since every such set satisfies $ \chi_K(p)\leq
\prod_{j=1}^{2n}\chi_{\pi_j(K)}(\pi_j(p))$, for all
$p\in\mathbb{R}^{2n+1}$, and on the other hand,
$\bigcap_{j=1}^{2n} \pi_j^{-1}(K_j)$ is compact in
$\mathbb{R}^{2n+1}$ whenever $K_1,\ldots,K_{2n}$ are compact
subsets of $\mathbb{R}^{2n}$, Theorem \ref{mainIntro2} is
equivalent to the statement that
\begin{equation}\label{form18-weak} \int_{\mathbb{R}^{2n+1}}
\prod_{j=1}^{2n}\chi_{K_j}(\pi_j(p))\;dp \lesssim
\prod_{j=1}^{2n}|K_j|^{\frac{n+1}{n(2n+1)}} \end{equation} holds
for all compact sets $K_1,\ldots,K_{2n}\subset \mathbb{R}^{2n}$.
Here we have identified, for $j=1,\ldots,2n$, the
$\{x_j=0\}$-plane in $\mathbb{R}^{2n+1}$ with $\mathbb{R}^{2n}$,
so that $\pi_1,\ldots,\pi_{2n}$ are now mappings from
$\mathbb{R}^{2n+1}$ to $\mathbb{R}^{2n}$. Using this expression,
it is evident that Theorem \ref{mainIntro2} follows from the next
result:

\begin{thm}\label{mainIntroStrong}
 Fix $n\in \mathbb{N}$. Then
\begin{equation}\label{form18Strong} \int_{\mathbb{R}^{2n+1}} \prod_{j=1}^{2n} f_j(\pi_j(p))\,dp \lesssim \prod_{j=1}^{2n}
\|f_j\|_{\frac{n(2n+1)}{n+1}},
\end{equation}
for all nonnegative Lebesgue measurable functions
$f_1,\ldots,f_{2n}$ on $\mathbb{R}^{2n}$.
\end{thm}
The coarea formula coupled with \eqref{eq:RankCondition} shows
that the preimages of Lebesgue null sets in $\mathbb{R}^{2n}$
under $\pi_j$ are Lebesgue null sets in $\mathbb{R}^{2n+1}$, and
so $f_j \circ \pi_j:\mathbb{R}^{2n+1}\to [0,+\infty]$ is Lebesgue
measurable under the assumptions of the theorem, and the integral
on the left-hand side of \eqref{form18Strong} makes sense.

The bilinear case ($n=1$) of Theorem \ref{mainIntroStrong} follows
directly from  the $L^{3/2}-L^3$ boundedness of the standard Radon
transform in $\mathbb{R}^2$, and as such was known -- by a
Fourier-analytic proof -- at least since the work of Oberlin and
Stein \cite{MR667786}; see Section \ref{s:H1}. Theorem
\ref{mainIntroStrong} for $n=1$ is also an instance of
\cite[Theorem 1.1]{MR4176542} (with $b=(2,2)$ in
\cite[(1.6)]{MR4176542}  and $(p_1,p_2)=(3/2,3/2)$ in
\cite[(1.8)]{MR4176542}). The corresponding weak-type bound
(Theorem \ref{mainIntro2} for $n=1$) was also obtained by Gressman
as a special case of the endpoint restricted weak-type estimates
in \cite[Theorem 2]{MR2576685}. Due to the nilpotent group
structure of the Heisenberg group and the invariance of the
problem under Heisenberg dilations, it is a particularly simple
instance of Gressman's more general theorem. The proofs in
\cite{MR2576685,MR4176542} used an adaptation of the \emph{method
of refinements}, which was initiated by Christ \cite{MR1654767} in
order to prove $L^p-L^q$ bounds for certain convolution-type
operators.

To the best of our knowledge, Theorem \ref{mainIntroStrong} for
$n>1$ has not appeared in the literature before.  Stovall proved
in \cite{MR2895344} similar  inequalities for multilinear
Radon-like transforms, but \eqref{form18Strong} for $n>1$
constitutes a strong-type endpoint case that is not covered by her
work. In her notation, our setting corresponds to
 $b(p)=((n+1)/n,\ldots,(n+1)/n)$, which is a point on the boundary
 of the polytope $P$ mentioned in \cite[Theorem 3]{MR2895344}.

Our approach to Theorem \ref{mainIntroStrong} can be applied to
prove something a bit more general, see Theorem
\ref{mainIntroStrong_nonperp} for the precise statement. The idea
is to apply the same inductive procedure and reduce the claim to
an $L^{3/2}$-$L^3$ boundedness statement for a certain operator in
the plane. In the case of  Theorem \ref{mainIntroStrong}, this
operator happens to be the standard Radon transform, but other
choices are possible as well, for instance convolution by a fixed
parabola in $\mathbb{R}^2$, cf.\ the use of \eqref{eq:ConvParab}
in connection with Example \ref{ex:operators}.

\medskip

It is easy to see that the exponents in the Heisenberg
Loomis-Whitney inequality \eqref{form18} are sharp by considering
boxes of the form $[-r,r]^{2n} \times [-r^{2},r^{2}]$. Besides the
difference in the definition of the projections $\tilde \pi_j$ and
$\pi_j$, there is another obvious difference between (the case $d
= 2n+1$ of) the standard Loomis-Whitney inequality \eqref{LWIneq},
and \eqref{form18}: the former bounds the volume of $K$ in terms
of $2n+1$ projections, and the latter in terms of only $2n$
projections. One might therefore ask: is there a version of
\eqref{LWIneq} for $2n$ orthogonal projections $\R^{2n+1} \to
\R^{2n}$ -- and does it look like \eqref{form18}? The answer is
negative. This is a very special case of \cite[Theorem
1.13]{MR2377493} (cf.\ also \cite{MR1726701,MR1969206,MR2895344}),
but perhaps it is illustrative to see an explicit computation for
$n=1$:
\begin{ex} Consider the two standard orthogonal coordinate projections $\tilde{\pi}_{1},\tilde{\pi}_{2}$ in $\R^{3}$
to the $x_2t$- and $x_1t$-planes. If $K = [0,1]^{2} \times
[0,\delta]$, then $|K| = \delta$, and also $|\tilde{\pi}_{1}(K)| =
\delta = |\tilde{\pi}_{2}(K)|$. So, for $\delta > 0$ small, an
inequality of the form
\begin{equation}\label{form19} |K| \lesssim |\tilde{\pi}_{1}(K)|^{\lambda} \cdot |\tilde{\pi}_{2}(K)|^{\lambda} \end{equation}
can only hold for $\lambda \leq \tfrac{1}{2}$.
On the other hand, if $K_{R} = [0,R]^{3}$, with $R \gg 1$, then $|K_{R}| = R^{3}$
and $|\tilde{\pi}_{1}(K_{R})| = R^{2} = |\tilde{\pi}_{2}(K_{R})|$,
so \eqref{form19} can only hold for $\lambda \geq \tfrac{3}{4}$.
The latter example naturally does not contradict \eqref{form18}: note that $|\pi_{j}(K_{R})| \sim R^{3}$ for $R \gg 1$. \end{ex}

\subsection{Gagliardo-Nirenberg-Sobolev inequalities in $\He^n$}

In $\R^{d}$, it is well-known that the Loomis-Whitney inequality
implies the \emph{Gagliardo-Nirenberg-Sobolev inequality}
\begin{equation}\label{GSN} \|f\|_{d/(d - 1)} \leq \prod_{j = 1}^{d} \|\partial_{j}f\|_{1}^{1/d}, \qquad f \in C^{1}_{c}(\R^{d}). \end{equation}
Similarly,  an $\He^n$-analogue of \eqref{GSN} can be obtained as
a corollary of Theorem \ref{mainIntro2}:
\begin{thm}\label{mainIntro3} Let $f \in BV(\He)$. Then,
\begin{equation}\label{GSNHe} \|f\|_{\frac{2n+2}{2n+1}} \lesssim \prod_{j=1}^{2n}\|X_jf\|^{\frac{1}{2n}}. \end{equation}
\end{thm}
Here
\begin{equation}\label{eq:Xj}
X_j = \partial_{x_j} - \tfrac{x_{n+j}}{2}\partial_{t} \quad
\text{and} \quad X_{n+j} = \partial_{x_{n+j}} +
\tfrac{x_j}{2}\partial_{t},\quad (j=1,\ldots,n),
\end{equation} are the standard left-invariant "horizontal"
vector fields in $\He^n$, and $BV(\He^n)$ refers to functions $f
\in L^{1}(\He^n)$ whose distributional $X_j$ derivatives are
signed Radon measures with finite total variation, denoted
$\|\cdot\|$.

Theorem \ref{mainIntro3} presents a sharper version of the
well-known "geometric" Sobolev inequality
\begin{equation}\label{sobolev} \|f\|_{\frac{2n+2}{2n+1}} \lesssim \|\nabla_{\He}f\|, \qquad f \in BV(\He^n), \end{equation}
proven by Pansu \cite{MR676380} for $n=1$ as a corollary of the
isoperimetric inequality in $\He^1$. Here $\nabla_{\He}f =
(X_1f,\ldots,X_{2n}f)$. Versions of geometric Sobolev inequalities
and isoperimetric inequalities were obtained in $\mathbb{H}^n$ and
even more general frameworks by several authors, for instance in
\cite{MR1312686,MR1404326}. A proof of \eqref{sobolev} for $n=1$,
using the fundamental solution of the sub-Laplace operator
$\bigtriangleup_{\He}$, is discussed in \cite[Section
5.3]{MR2312336}, following the approach of \cite{MR1312686}. On
the other hand, Theorem \ref{mainIntro3} can be derived from
Theorem \ref{mainIntro2}. This deduction follows a standard
argument, but we present it here to highlight the fact that the
geometric Sobolev and isoperimetric inequalities in all Heisenberg
groups are ultimately based on planar geometry and they can be
deduced from boundedness properties of the Radon transform in
$\mathbb{R}^2$.

Theorem \ref{mainIntroStrong} is related to \emph{Brascamp-Lieb
inequalities}. We direct the reader to e.g.
\cite{2018arXiv181111052B,MR2377493,MR412366} and the references
therein. Euclidean Loomis-Whitney and Brascamp-Lieb inequalities
can be proven by the technique of \emph{heat flow monotonicity},
see \cite{MR2377493}. The same approach has been attempted in
Carnot groups by Bramati \cite{Bramati}, but there seems to be a
gap in the argument, which has been confirmed with the author.
More precisely, the exponents appearing in the proof of
\cite[Theorem 3.2.3]{Bramati} have not been chosen consistently.
It remains an open problem to see whether the Loomis-Whitney
inequalities in Carnot groups can be obtained by the heat flow
approach.

\medskip

\textbf{Structure of the paper.} In Section \ref{s:H1}, we explain
how Theorems \ref{mainIntro2} and \ref{mainIntroStrong} for $n=1$
follow from known $L^p$ improving properties of the Radon
transform in $\mathbb{R}^2$. In Section \ref{s:Hn}, we deduce
Theorems \ref{mainIntro2} and \ref{mainIntroStrong} for arbitrary
$n>1$ by induction from the corresponding inequalities in
$\mathbb{H}^1$. In Section \ref{s:Appl}, we show how to derive the
Gagliardo-Nirenberg-Sobolev inequality, Theorem \ref{mainIntro3},
as an application of the Loomis-Whitney inequality in
$\mathbb{H}^n$. Finally, in Section \ref{s:Generalized} we explain
how to adapt the approach from Section \ref{s:Hn} to prove the
generalized Loomis-Whitney-type inequality stated in Theorem
\ref{mainIntroStrong_nonperp}.


\section{Inequalities in the first Heisenberg group}\label{s:H1}

In this section, we review the proof for the Loomis-Whitney
inequality in the \emph{first} Heisenberg group. For this purpose
is more convenient to use slightly different notation. In
particular, points in $\mathbb{R}^3$ will be denoted by $(x,y,t)$
(instead of $(x,t)=(x_1,x_2,t))$. The group product of
$\mathbb{H}^1$ then reads in coordinates as follows:
\begin{equation}\label{eq:GroupProd1} (x,y,t) \cdot (x',y',t') := (x + x', y + y', t + t' + \tfrac{1}{2}(xy' - yx')). \end{equation}
The vertical Heisenberg projections to the $yt$- and the
$xt$-plane, respectively, are explicitly given by
\begin{displaymath} \pi_{1}(x,y,t) = (0,y,t + \tfrac{xy}{2}) \quad \text{and} \quad \pi_{2}(x,y,t) = (x,0,t - \tfrac{xy}{2}). \end{displaymath}
We recall the statement of Theorems \ref{mainIntro2} and
\ref{mainIntroStrong} for $n=1$:

\begin{thm}[Loomis-Whitney inequality in $\mathbb{H}^1$]\label{main} Let $K \subset \He^1$ be arbitrary. Then,
\begin{equation}\label{form1} |K| \lesssim |\pi_{1}(K)|^{2/3} \cdot |\pi_{2}(K)|^{2/3}. \end{equation}
\end{thm}

\begin{thm}\label{mainStrong} For all nonnegative Lebesgue
measurable functions $f_1$ and $f_2$ on $\mathbb{R}^2$ it holds
that
\begin{equation}\label{form1Strong}
\int_{\mathbb{R}^3} f_1(\pi_1(p)) f_2(\pi_2(p))\,dp \lesssim
\|f_1\|_{\frac{3}{2}} \|f_2\|_{\frac{3}{2}}.
 \end{equation}
\end{thm}

On the left-hand side of \eqref{form1}, the notation "$|\cdot|$"
refers to Lebesgue outer measure on $\R^{3}$. Similarly, on the
right-hand side of \eqref{form1}, the notation "$|\cdot|$" refers
to Lebesgue outer measure on $\R^{2}$. Clearly, Theorem
\ref{mainStrong} implies  Theorem \ref{main}. We now explain how
Theorem \ref{mainStrong} itself follows directly from known
$L^p$-improving properties of the standard Radon transform in the
plane $\mathbb{R}^2$.

\medskip

Let $S^1$ be the unit sphere in $\mathbb{R}^2$.
{For a smooth, compactly supported function $f$ on $\mathbb{R}^2$,
the \emph{Radon transform} (or \emph{$X$-ray transform}) $Rf$ is
defined by}
\begin{equation}\label{eq:DefRf}
Rf(\sigma,s):= \int_{\langle z,\sigma\rangle = s} f(z)\,dz,\quad
(\sigma,s)\in S^1\times \mathbb{R}.
\end{equation}
Here $dz$ is the $1$-dimensional Lebesgue measure on the line
$\{z\in \mathbb{R}^2:\, \langle z,\sigma\rangle = s\}$.  Using
Fourier analysis (notably Plancherel's theorem) and complex
interpolation, Oberlin and Stein \cite{MR667786} proved that $R$
extends to a bounded operator from $L^{3/2}(\mathbb{R}^2)$ to
$L^3(S^1 \times \mathbb{R})$. Their result is more general, but
this is the only information one needs to deduce Theorem
\ref{mainStrong}.

\medskip

The connection between inequality \eqref{form1Strong}  and the
Radon transform is illustrated by the formula
\begin{equation}\label{eq:formula}
\int_{\mathbb{R}^3} f_1(\pi_1(p))f_2(\pi_2(p))\;dp =
\int_{\mathbb{R}^2}R\left(f_1\right)(\sigma(x),s_{x,t})f_2(x,t)\,\frac{d(x,t)}{\sqrt{1+x^2}}
\end{equation}
 with
$s_{x,t}= t/\sqrt{1+x^2}$ and $\sigma(x):=
\frac{1}{\sqrt{1+x^2}}(-x,1)$ for smooth compactly supported
functions $f_1$ and $f_2$ on $\mathbb{R}^2$. The proof of
inequality \eqref{form1Strong} using the result in \cite{MR667786}
is an instance of a more general phenomenon that relates
$L^p$-improving properties of averaging operators along curves to
inequalities of the form \eqref{form1Strong} with two factors in
the integral. The general framework is explained in detail in
\cite[9.5. Double fibration formulation]{MR1726701} and
\cite[Section 1]{MR1969206}. For our purpose it is convenient to
work with a linear operator $T$ that yields functions on
$\mathbb{R}^2$, rather than $S^1\times \mathbb{R}$ as in the case
of the Radon transform, so instead of applying directly
\eqref{eq:formula}, we will pass via an identity of the form
\begin{equation*}
\int_{\mathbb{R}^3} f_1(\pi_1(p))f_2(\pi_2(p))\;dp =
\int_{\mathbb{R}^2}Tf_1(x,t)f_2(x,t)\,d(x,t);
\end{equation*}
see the proof of Theorem \ref{mainStrong}. For smooth, compactly
supported functions $f$ on $\mathbb{R}^2$, we define
\begin{equation}\label{eq:defT_for_n=1}
Tf(x,t):= \int_{\mathbb{R}} f(y,t+xy)\,dy,\quad (x,t)\in
\mathbb{R}^2.
\end{equation}
The next statement follows immediately from \cite{MR667786} by
relating the operator $T$ to the Radon transform $R$, and we do
not claim any novelty for it.
\begin{thm}\label{t:T_n=1_bdd} There exists a constant $C$
such that the operator $T$ defined in \eqref{eq:defT_for_n=1}
satisfies
\begin{displaymath}
\|Tf\|_{3}\leq C \|f\|_{\frac{3}{2}}
\end{displaymath}
for all smooth, compactly supported functions $f$.
\end{thm}

\begin{proof}
We reduce Theorem \ref{t:T_n=1_bdd} to a statement about the Radon
transform that was proven in \cite{MR667786}. We fix a smooth
compactly supported function $f$ and start by writing
\begin{align}
\|Tf\|_3&=\left[\int_{\mathbb{R}^2} \left|\int_{\mathbb{R}} f(y,t+xy)\,dy\right|^{3}d(x,t)\right]^{\frac{1}{3}}\\
&= \left[\int_{\mathbb{R}^2} \left|\int_{\mathbb{R}}
f(y,t+xy)\sqrt{1+x^2}\,dy\right|^{3}\frac{d(x,t)}{(1+x^2)^{3/2}}\right]^{\frac{1}{3}}\notag\\
&=\left[ \int_{\mathbb{R}^2}\left|\int_{\ell_{x,t}}
f\,d\lambda_{\ell_{x,t}}\right|^3\,\frac{d(x,t)}{(1+x^2)^{3/2}}\right]^{\frac{1}{3}}.\label{eq:star}
\end{align}
Here $d\lambda_{\ell_{x,t}}$ denotes the $1$-dimensional Lebesgue
measure on the line
\begin{displaymath}
\ell_{x,t}:= \left\{z\in\mathbb{R}^2:\, \langle z,\sigma(x)\rangle
= \frac{t}{\sqrt{1+x^2}} \right\} = \{(y,t+ xy):\; y\in
\mathbb{R}\}\text{ with }\sigma(x):=
\frac{1}{\sqrt{1+x^2}}\begin{pmatrix}-x\\1\end{pmatrix}.
\end{displaymath}
Thus, recalling the definition of the Radon transform in
\eqref{eq:DefRf}, we obtain from \eqref{eq:star} that
\begin{align*}
\|Tf\|_3 &= \left[
\int_{\mathbb{R}^2}|Rf(\sigma(x),s_{x,t})|^3\,\frac{d(x,t)}{(1+x^2)^{3/2}}\right]^{\frac{1}{3}}\\
&= \left[
\int_{\mathbb{R}}\left(\int_{\mathbb{R}}|Rf(\sigma(x),s_{x,t})|^3\,\frac{dt}{\sqrt{1+x^2}}\right)\frac{dx}{1+x^2}\right]^{\frac{1}{3}}
\end{align*}
with $s_{x,t}= t/\sqrt{1+x^2}$. Changing variables in the inner
integral, and observing that $x\mapsto \sigma(x)$ parameterizes an
arc in $S^1$, we then deduce that
\begin{displaymath}
\|Tf\|_3 =\left[
\int_{\mathbb{R}}\left(\int_{\mathbb{R}}|Rf(\sigma(x),s)|^3\,ds\right)|\sigma'(x)|\,dx\right]^{\frac{1}{3}}\leq
\left[
\int_{S^1}\left(\int_{\mathbb{R}}|Rf(\sigma,s)|^3\,ds\right)\,d\sigma\right]^{\frac{1}{3}}=\|Rf\|_{3},
\end{displaymath}
where $\sigma$ denotes the usual Lebesgue (arc-length) measure on
$S^1$. Now the theorem follows from the inequality $\|Rf\|_3 \leq
C \|f\|_{\frac{3}{2}}$ for the Radon transform, which was
established as a special case of \cite[Theorem 1]{MR667786}.
\end{proof}

Theorem \ref{mainStrong} is an immediate corollary of  Theorem
\ref{t:T_n=1_bdd}.

\begin{proof}[Proof of Theorem \ref{mainStrong}]
It suffices to prove the theorem for nonnegative smooth, compactly
supported functions on $\mathbb{R}^2$. Indeed, if $f_1$ is an
arbitrary nonnegative Lebesgue measurable function on
$\mathbb{R}^2$, we take a sequence $(f_{1,k})_{k\in\mathbb{N}}$ of
nonnegative $\mathcal{C}^{\infty}_{c}$ functions which converges
to $f_1$ with respect to $\|\cdot\|_{3/2}$ and pointwise almost
everywhere. In the same way, we approximate a given nonnegative
Lebesgue measurable function $f_2$ by a sequence
$(f_{2,k})_{k\in\mathbb{N}}$ of nonnegative
$\mathcal{C}^{\infty}_{c}$ functions. Then, assuming that the
theorem holds for nonnegative $\mathcal{C}^{\infty}_{c}$
functions, we apply it to the pair $f_{1,k},f_{2,k}$ for every
$k\in \mathbb{N}$. The desired inequality \eqref{form1Strong} for
the functions $f_1,f_2$ follows by Fatou's lemma, observing that
for $j\in \{1,2\}$, the sequence $(f_{j,k}\circ \pi_j)_{k\in
\mathbb{N}}$ converges pointwise almost everywhere to $(f_j\circ
\pi_j)_{j\in \mathbb{N}}$ since the preimage of a Lebesgue null
set in $\mathbb{R}^2$ is a Lebesgue null set in $\mathbb{R}^3$,
according to the remark below Theorem \ref{mainIntroStrong}.

We now prove the theorem for nonnegative $\mathcal{C}^{\infty}_c$
functions on $\mathbb{R}^2$. Let $f_1$ and $f_2$ be such functions
and let us prove that they satisfy the inequality
\eqref{form1Strong}. To this end, we rewrite  the left-hand side
using the volume-preserving diffeomorphism
\begin{displaymath}
\Phi:\mathbb{R}^3 \to \mathbb{R}^3,\quad \Phi(x,y,t)=(x,0,t)\cdot
(0,y,0)=\left(x,y,t+\tfrac{1}{2}x y\right).
\end{displaymath}
With this definition,
\begin{displaymath}
\pi_1(\Phi(x,y,t))=(y,t+xy)\quad\text{and}\quad\pi_2(\Phi(x,y,t))=(x,t)
\end{displaymath}
for all $(x,y,t)\in \mathbb{R}^3$. Hence the left-hand side of
\eqref{eq:formula} can be expressed as follows:
\begin{align*}
\int_{\mathbb{R}^3} f_1(\pi_1(p))f_2(\pi_2(p))\;dp & =
\int_{\mathbb{R}^2}\int_{\mathbb{R}}f_1(\pi_1(\Phi(x,y,t))f_2(\pi_2(\Phi(x,y,t)))\,
dy \,d(x,t)\\
&=
\int_{\mathbb{R}^2}\left(\int_{\mathbb{R}}f_1(y,t+xy)\,dy\right)
f_2(x,t)\,d(x,t)\\
&= \int_{\mathbb{R}^2}Tf_1(x,t) f_2(x,t)\,d(x,t),
\end{align*}
using the linear operator $T$ defined in \eqref{eq:defT_for_n=1}.
Thus, it follows from H\"older's inequality with exponents $p=3$
and $p'=3/2$, and the mapping property of $T$ stated in Theorem
\ref{t:T_n=1_bdd}, that
\begin{displaymath}
\int_{\mathbb{R}^3} f_1(\pi_1(p))f_2(\pi_2(p))\;dp \leq \|Tf_1\|_3
\|f_2\|_{\frac{3}{2}}\leq C\|f_1\|_{\frac{3}{2}}
\|f_2\|_{\frac{3}{2}},
\end{displaymath}
as desired.
\end{proof}

\section{Inequalities in higher-dimensional Heisenberg groups}\label{s:Hn}

In this section we prove Theorem \ref{mainIntroStrong} for
arbitrary $n>1$ by induction, using Theorem \ref{mainStrong} as a
base case. To be precise, instead of directly aiming at inequality
\eqref{form18Strong} in Theorem \ref{mainIntroStrong}, we will
prove Theorem \ref{t:induction theorem} first. Its statement
reflects the algebraic structure of the Heisenberg group. In
brief, for a fixed   $k\in \{1,\ldots,n\}$, the different Lebesgue
exponents on the right-hand side of \eqref{eq:vertex_eq} appear by
applying once the commutator relation $[X_k,X_{n+k}]=\partial_t$,
where $X_k$ and $X_{n+k}$ are defined as  in \eqref{eq:Xj}. This
is done by employing the strong-type bound for $\mathbb{H}^1$
given by Theorem \ref{mainStrong}. After this initial step, the
 remaining steps  of the induction use only
standard properties of integrals and elementary estimates by
H\"older's and Minkowski's integral inequalities.

\begin{thm}\label{t:induction theorem}
Fix $n\in \mathbb{N}$. Then, for all nonnegative Lebesgue
measurable functions $f_1,\ldots,f_{2n}$ on $\mathbb{R}^{2n}$, we
have
\begin{equation}\label{eq:vertex_eq}
\int_{\mathbb{R}^{2n+1}} \prod_{j=1}^{2n} f_j(\pi_j(p))\;dp
\lesssim \|f_k\|_{\frac{2n+1}{2}}\|f_{n+k}\|_{\frac{2n+1}{2}}
\prod_{\substack{j=1\\j\neq k}}^n\left(
\|f_j\|_{2n+1}\,\|f_{n+j}\|_{2n+1}\right),\quad k\in
\{1,\ldots,n\},
\end{equation}
with an implicit constant that may depend on $n$. For $n=1$, the
right-hand side of \eqref{eq:vertex_eq} equals $
\|f_1\|_{\frac{3}{2}}\|f_{2}\|_{\frac{3}{2}}$.
\end{thm}
The Lebesgue exponents in Theorem \ref{t:induction theorem}
correspond to vertex points on the boundary of the Newton polytope
in \cite[Section 3]{MR2895344} and as such are not covered by
\cite[Theorem 3]{MR2895344}.  For instance, the exponents in
\eqref{eq:vertex_eq} for $k=1<n$ corresponds to
$b(p)=(2,1,\ldots,1,2,1,\ldots,1)$ in the notation of
\cite[(2.5)]{MR2895344}.

For $n=1$, the statements of Theorem \ref{t:induction theorem} and
Theorem \ref{mainIntroStrong} are equivalent. For $n>1$, Theorem
\ref{t:induction theorem}, \eqref{eq:vertex_eq}, consists of $n$
separate inequalities. Knowing that they all hold for all
nonnegative measurable functions, one can deduce the inequality
\begin{equation}\label{eq:strong_type_bound_goal}
\int_{\mathbb{R}^{2n+1}} \prod_{j=1}^{2n} f_j(\pi_j(p))\;dp
\lesssim \prod_{j=1}^{2n} \|f_j\|_{\frac{n(2n+1)}{n+1}}
\end{equation}
postulated in Theorem \ref{mainIntroStrong} by multilinear
interpolation, as we will explain below the next remark.

\begin{remark}
If one is only interested in the Loomis-Whitney inequality in
$\mathbb{H}^n$ (Theorem \ref{mainIntro2}), and not in the
strong-type bound stated in Theorem \ref{mainIntroStrong}, then
one can finish the proof without using multilinear interpolation.
In particular all the geometric consequences that we list in
Section \ref{s:Appl} can be obtained by this simpler argument.
Indeed, let $K\subset \mathbb{R}^{2n+1}$ be a compact set. Then
Theorem \ref{t:induction theorem} implies that
\begin{equation*}
|K| \lesssim |\pi_k(K)|^{\frac{2}{2n+1}}
|\pi_{n+k}(K)|^{\frac{2}{2n+1}} \prod_{\substack{j=1\\j\neq
k}}^n\left(
|\pi_j(K)|^{\frac{1}{2n+1}}\,|\pi_{n+j}(K)|^{\frac{1}{2n+1}}\right).
\end{equation*}
for all $k\in \{1,\ldots,n\}$.
 Multiplying these $n$ inequalities together, we obtain
\begin{displaymath}
|K|^n \lesssim  \prod_{j=1}^{2n}|\pi_j(K)|^{\frac{n+1}{2n+1}},
\end{displaymath}
from where the Loomis-Whitney inequality in $\mathbb{H}^n$ follows
by taking the $n$-th root.
\end{remark}

To prove Theorem \ref{mainIntroStrong}, we will rephrase Theorem
\ref{t:induction theorem} by duality as bounds of the type
\begin{equation}\label{eq:OpAssk}
\|T(f_1,\ldots,f_{2n-1})\|_{q_k} \lesssim \prod_{j=1}^{2n-1}
\|f_j\|_{p_{j,k}},\quad \text{ for }k=1,\ldots,n,
\end{equation}
for a certain multilinear operator $T$. Then multilinear
interpolation will allow us to deduce the bound
\begin{equation}\label{eq:GoalOperatorBound}
\|T(f_1,\ldots,f_{2n-1})\|_{q}\lesssim \prod_{j=1}^{2n-1}
\|f_j\|_{p_j}
\end{equation}
with
\begin{equation}\label{eq:InterpolationQExponents}
\frac{1}{q} =  \frac{1}{n} \sum_{k=1}^n \frac{1}{q_k}, \quad
\text{and}\quad \frac{1}{p_j}= \frac{1}{n}\sum_{k=1}^n
\frac{1}{p_{j,k}},\quad j=1,\ldots, 2n-1.
\end{equation}
Finally, \eqref{eq:GoalOperatorBound} will yield
\eqref{eq:strong_type_bound_goal}. Before turning to the details,
we state the multilinear interpolation theorem which will be
applied repeatedly to infer \eqref{eq:GoalOperatorBound} from
\eqref{eq:OpAssk}. It can be proven by the method of complex
interpolation \cite{MR167830,MR0482275} and we simply state here a
version that is useful for our purposes. The theorem is formulated
for \emph{finitely simple functions} on a measure space. These are
functions of the form $\sum_{i=1}^N c_i \chi_{E_i}$ with the
requirement that $E_i$ is a measurable set of finite mass. In our
application, the relevant measure spaces will all be equal to
$\mathbb{R}^{2n}$ with the Lebesgue measure.

\begin{thm}[Corollary 7.2.11 in \cite{MR3243741}]\label{t:Grafakos} Assume that $T$
is an $m$-linear operator on the $m$-fold product of spaces of
finitely simple functions of $\sigma$-finite measure spaces
$(Y_j,\mu_j)$, and suppose that $T$ takes values in the set of
measurable functions of a $\sigma$-finite measure space $(Z,\nu)$.
Let $1\leq p_{1,j},p_{2,j},q_1,q_2\leq \infty$ for all $1\leq
j\leq m$, $0<\theta<1$. Suppose that for all finitely simple $f_j$
on $Y_j$ one has
\begin{displaymath}
\|T(f_1,\ldots,f_m)\|_{q_1}\leq M_1
\prod_{j=1}^m\|f_j\|_{p_{1,j}}\quad\text{and}\quad
\|T(f_1,\ldots,f_m)\|_{q_2}\leq M_2\prod_{j=1}^m
\|f_j\|_{p_{2,j}}.
\end{displaymath}
Then for all finitely simple functions $f_j$ on $Y_j$ it holds
that
\begin{displaymath}
\|T(f_1,\ldots,f_m)\|_q \leq M_1^{1-\theta} M_2^{\theta}
\prod_{j=1}^m\|f_j\|_{p_j},
\end{displaymath}
where
\begin{displaymath}
\frac{1}{q}=\frac{1-\theta}{q_1}+\frac{\theta}{q_2}\quad\text{and}\quad
\frac{1}{p_j}=\frac{1-\theta}{p_{1,j}}+\frac{\theta}{p_{2,j}}\quad\text{for
}j=1,\ldots,m.
\end{displaymath}
\end{thm}

Recalling Theorem \ref{mainStrong}, it suffices to prove  Theorem
\ref{mainIntroStrong} for $n>1$.

\begin{proof}[Proof of Theorem \ref{mainIntroStrong} for $n>1$ using Theorem \ref{t:induction theorem}]
Assume that the statement of Theorem \ref{t:induction theorem}
holds for a fixed natural number $n>1$. Our aim is to verify
\eqref{eq:strong_type_bound_goal} for all nonnegative measurable
functions $f_1,\ldots,f_{2n}$ on $\mathbb{R}^{2n}$. The desired
inequality can be spelled out as follows:
\begin{equation}\label{eq:DesiredIneqExplicit}
\int_{\mathbb{R}^{2n+1}} \prod_{j=1}^{n} \left(f_j(\hat
x_j,t+\tfrac{1}{2}x_j x_{n+j})\,f_{n+j}(\hat
x_{n+j},t-\tfrac{1}{2}x_j x_{n+j})\right)\;d(x,t) \lesssim
\prod_{j=1}^{2n} \|f_j\|_{\frac{n(2n+1)}{n+1}}.
\end{equation}
Here we have used the same notational convention as at the
beginning of Section \ref{ss:Heis}. The coordinate expressions
appearing in \eqref{eq:DesiredIneqExplicit} help us to define a
multilinear operator $T$ for which a bound of the type
\eqref{eq:GoalOperatorBound} will yield
\eqref{eq:DesiredIneqExplicit}. The idea is, essentially, to
express the left-hand side of \eqref{eq:DesiredIneqExplicit} as
the pairing of $T(f_1,\ldots,f_{2n-1})$ with $f_{2n}$, similarly
as we did in the proof of Theorem \ref{mainStrong}. To bring the
integral into this form, we first apply the Fubini-Tonelli theorem
and then the change of variables $\tau= t-\frac{1}{2}x_nx_{2n}$ in
the $t$-coordinate so that the left-hand side of
\eqref{eq:strong_type_bound_goal} equals
\begin{align}
&\int_{\mathbb{R}^{2n+1}} \prod_{j=1}^{2n} f_j(\pi_j(p))\;dp
\label{id:Op}\\&= \int_{\mathbb{R}^{2n}}
\left[\int_{\mathbb{R}}f_n(\hat x_n,\tau+x_nx_{2n})
\prod_{\substack{j=1\\j\neq n}}^{2n-1}
f_j(\pi_j(x,\tau+\tfrac{1}{2}x_nx_{2n}))\,dx_{2n}\right]
f_{2n}(\hat x_{2n},\tau)\; d(\hat x_{2n},\tau).\notag
\end{align}
This identity motivates the following definition of the operator
$T$. For all finitely simple functions $g_1,\ldots,g_{2n-1}$ on
$\mathbb{R}^{2n}$, we define
\begin{align*}\label{def:T}
T (g_1,\ldots,g_{2n-1})(\hat x_{2n},\tau)
:=\int_{\mathbb{R}}g_n(\hat x_n,\tau+x_nx_{2n})
\prod_{\substack{j=1\\j\neq n}}^{2n-1} g_j(\pi_j(x,
,\tau+\tfrac{1}{2}x_nx_{2n}))\,dx_{2n}.\notag
\end{align*}
Using \eqref{id:Op}, and applying H\"older's inequality with
exponents $n(2n+1)/(n+1)$ and its dual exponent
\begin{equation}\label{eq:q}
q:=\frac{n(2n+1)}{2n^2-1}, \end{equation} we find for all
nonnegative finitely simple functions $f_1,\ldots,f_{2n-1}$ and
all nonnegative measurable functions $f_{2n}$ that
\begin{align*}
\int_{\mathbb{R}^{2n+1}} \prod_{j=1}^{2n} f_j(\pi_j(p))\;dp& =
\int_{\mathbb{R}^{2n}} T(f_1,\ldots,f_{2n-1})(w) f_{2n}(w)\; dw
\\&\leq \| T(f_1,\ldots,f_{2n-1})\|_q\, \|f_{2n}\|_{\frac{n(2n+1)}{n+1}}.
\end{align*}
Hence, to prove \eqref{eq:strong_type_bound_goal} for such
functions $f_1,\ldots,f_{2n-1}$, we aim to show
\begin{equation}\label{eq:GoalOperatorBound_inproof}
\|T(f_1,\ldots,f_{2n-1})\|_{q}\lesssim \prod_{j=1}^{2n-1}
\|f_j\|_{p_j},\quad \text{for
}p_1=\ldots=p_{2n-1}=\frac{n(2n+1)}{n+1}
\end{equation}
and $q$ as in \eqref{eq:q}. Having established
\eqref{eq:strong_type_bound_goal} for nonnegative finitely simple
functions, it is straightforward to obtain the inequality also for
all nonnegative measurable functions $f_1,\ldots,f_{2n}$. Indeed,
given nonnegative measurable functions $f_1,\ldots,f_{2n}$, we may
assume that the right-hand side of
\eqref{eq:strong_type_bound_goal} is finite, and then each $f_j$
is the pointwise almost everywhere limit of an increasing sequence
of nonnegative finitely simple functions that converge to $f_j$
also in $\|\cdot\|_{p_j}$-norm, and Theorem \ref{mainIntroStrong}
follows, by an analogous argument as described at the beginning of
the proof of Theorem \ref{mainStrong}.

 Thus it remains to prove the claim
\eqref{eq:GoalOperatorBound_inproof} for nonnegative finitely
simple functions. It may be illustrative to compare this with the
bound for the linear operator $T$ in Theorem \ref{t:T_n=1_bdd},
which is essentially the case $n=1$ of what we aim to prove,
albeit stated for smooth and compactly supported functions.

For $n>1$, we will deduce \eqref{eq:GoalOperatorBound_inproof}
from Theorem \ref{t:induction theorem}. Recall that the left-hand
sides of the inequalities in \eqref{eq:vertex_eq} can be expressed
as pairings of $T(f_1,\ldots,f_{2n-1})$ with $f_{2n}$, according
to the formula \eqref{id:Op} and the definition of $T$ if
$f_1,\ldots,f_{2n-1}$ are nonnegative finitely simple functions
and $f_{2n}$ is an arbitrary nonnegative measurable function. Then
the inequalities stated in \eqref{eq:vertex_eq} for $k=1,\ldots,n$
imply by duality that
\begin{equation}\label{eq:OpAssk_inproof}
\|T(f_1,\ldots,f_{2n-1})\|_{q_k} \lesssim \prod_{j=1}^{2n-1}
\|f_j\|_{p_{j,k}},\quad \text{ for }k=1,\ldots,n,
\end{equation}
{for all nonnegative finitely simple functions
$f_1,\ldots,f_{2n-1}$ on $\mathbb{R}^{2n}$,} and exponents
\begin{displaymath}
q_k=\left\{\begin{array}{ll}(2n+1)/(2n),&k=1,\ldots,n-1,
\\
(2n+1)/(2n-1),&k=n,\end{array}\right.
\end{displaymath}
and
\begin{displaymath}p_{j,k}= \left\{\begin{array}{ll}2n+1,&k\notin
\{j,j+n,j-n\},\\(2n+1)/2,&k\in \{j,j+n,j-n\}
\end{array}\right.,\quad j=1,\ldots,2n-1,\quad k=1,\ldots,n.
\end{displaymath}  Here, for
every $k=1,\ldots,n$, we take the Lebesgue exponent associated to
the $f_{2n}$-term on the right-hand side of the corresponding
inequality in \eqref{eq:vertex_eq}, and we let $q_k$ be the dual
of that exponent. This explains why the formula for $q_n$ is
different from $q_1=\ldots=q_{n-1}$. The exponent $p_{j,k}$ is
simply the Lebesgue exponent of the $f_j$-term that appears in the
$k$-th inequality of \eqref{eq:vertex_eq}.

The key property of the exponents in
\eqref{eq:GoalOperatorBound_inproof} and \eqref{eq:OpAssk_inproof}
is that they are related by convex combinations as indicated in
\eqref{eq:InterpolationQExponents}. Indeed, we compute that
\begin{equation*}
\frac{1}{q} = \frac{2n^2-1}{n(2n+1)}
=\frac{1}{n}\frac{2n-1}{2n+1}+ \sum_{k=1}^{n-1} \frac{1}{n}
\frac{2n}{2n+1}= \sum_{k=1}^n \frac{1}{n} \frac{1}{q_k},
\end{equation*}
 and similarly,
\begin{equation*}
\frac{1}{p_j}= \frac{n+1}{n(2n+1)} = \frac{1}{n}
\frac{2}{2n+1}+\sum_{\substack{k=1\\k\notin \{j,n-j\}}}^n
\frac{1}{n} \frac{1}{2n+1} = \sum_{k=1}^n
\frac{1}{n}\frac{1}{p_{j,k}},\quad j=1,\ldots, 2n-1.
\end{equation*}

To conclude the proof, we apply multilinear interpolation. Theorem
\ref{t:Grafakos} allows us to interpolate between two operator
bounds. In order to deduce \eqref{eq:GoalOperatorBound_inproof}
from the family of $n$ operator bounds stated in
\eqref{eq:OpAssk_inproof}, we apply Theorem \ref{t:Grafakos} for
$m=2n-1$ iteratively $(n-1)$-times, noting that
\eqref{eq:OpAssk_inproof} also holds for finitely simple
functions, as required by Theorem \ref{t:Grafakos}. The specific
form of the exponents is not used in this argument, we only need
to know that we are dealing with convex combinations as in
\eqref{eq:InterpolationQExponents}, and observe the identity
\begin{displaymath}
\tfrac{1}{k}[a_1+\cdots+ a_{k}] =
\left(1-\tfrac{1}{k}\right)\left(\tfrac{1}{k-1}[a_1+\cdots+a_{k-1}]\right)+\tfrac{1}{k}
a_k,
\end{displaymath}
for $k>1$, which allows to obtain
\eqref{eq:GoalOperatorBound_inproof} by successive interpolation.
More precisely, we apply first Theorem \ref{t:Grafakos} for
$\theta=\frac{1}{2}$ to the two operator bounds given by
\eqref{eq:OpAssk_inproof} for $k=1$ and $k=2$. Then we apply
Theorem \ref{t:Grafakos} with $\theta=\frac{1}{3}$ to interpolate
between this newly obtained bound and the operator bound stated in
\eqref{eq:OpAssk_inproof} for $k=3$. We continue until, in the
last step, we apply the theorem with $\theta=\frac{1}{n}$ to
interpolate between the previously obtained bound and the bound
for $k=n$. This yields \eqref{eq:GoalOperatorBound_inproof}  for
all nonnegative finitely simple functions $f_1,\ldots,f_{2n-1}$,
and thus concludes the proof of the theorem.
\end{proof}

\begin{proof}[Proof of Theorem \ref{t:induction theorem}]
First, by the same reasoning as at the beginning of the proof of
Theorem \ref{mainStrong}, it suffices to verify the claim for
nonnegative, smooth, and compactly supported functions.

We  fix $n\in \mathbb{N}$, $n>1$, and assume that the statement of
Theorem \ref{t:induction theorem} has already been proven for all
natural numbers from $1$ to $n-1$. Recall that the base case of
this induction is the content of Theorem \ref{mainStrong}. Given
nonnegative $\mathcal{C}^{\infty}_{c}$ functions
$f_1,\ldots,f_{2n}$, we now aim to show the $n$ inequalities
stated in \eqref{eq:vertex_eq}. We will explain the details only
for $k=1$, as the other inequalities can be proven in exactly the
same manner.

Throughout the following computation, points in
$\mathbb{R}^{2n+1}$ will be denoted in coordinates by $(x,t)$ with
$x\in \mathbb{R}^{2n}$ and $t\in \mathbb{R}$. For $1\leq i<2n$, we
also write $\hat{x}_{j_1,\ldots,j_i}$ to denote the point in
$\mathbb{R}^{2n-i}$ that is obtained by deleting the
$j_1,\ldots,j_i$-th coordinates of $x$.

First, we apply the Fubini-Tonelli theorem and then the
transformation $t\mapsto t-\tfrac{1}{2}x_n x_{2n} = \tau$ in the
inner integral:
\begin{align*}
I:= &\int_{\mathbb{R}^{2n}} \int_{\mathbb{R}} \prod_{j=1}^{2n}
f_j(\pi_j(x,t))\,dt\,dx= \int_{\mathbb{R}^{2n}} \int_{\mathbb{R}}
\prod_{j=1}^{2n}f_j(\pi_j(x,\tau+\tfrac{1}{2}x_n x_{2n}))\,d\tau dx\\
= &\int_{\mathbb{R}^{2n}} \int_{\mathbb{R}}  f_n(\hat x_n,\tau +
x_n x_{2n}) f_{2n}(\hat x_{2n}, \tau) \prod_{\substack{j=1\\j\neq
n,2n}}^{2n}f_j(\pi_j(x,\tau+\tfrac{1}{2}x_n x_{2n}))\,d\tau \,dx\\
= & \int_{\mathbb{R}^{2n}}  f_{2n}(\hat x_{2n}, \tau)\left[
\int_{\mathbb{R}}  f_n(\hat x_n,\tau + x_n x_{2n})
\prod_{\substack{j=1\\j\neq
n,2n}}^{2n}f_j(\pi_j(x,\tau+\tfrac{1}{2}x_n
x_{2n}))\,dx_{2n}\right]\,d(\hat x_{2n},t).
\end{align*}
Here, $d\hat x_{2n}= dx_1\ldots dx_{2n-1}$, and similar notation
will be used also below. The change of variables was motivated by
the observation that
\begin{displaymath}\pi_{2n}\left(x,\tau+\tfrac{1}{2}x_n x_{2n}\right)=(\hat
x_{2n},\tau),\end{displaymath} so that the $f_{2n}$-term becomes
independent of the $2n$-th coordinate of $x$. Applying H\"older's
inequality with exponents $p=2n+1$ and $p'=(2n+1)/2n$, we can
split this factor off to obtain $I\leq \|f_{2n}\|_{2n+1}\, J$ with
\begin{align*}
J:=\left[\int_{\mathbb{R}^{2n}} \left(\int_{\mathbb{R}}f_n(\hat
x_n,\tau + x_n x_{2n})\prod_{\substack{j=1\\j\neq
n,2n}}^{2n}f_j(\pi_j(x,\tau+\tfrac{1}{2}x_n x_{2n}))\,d
x_{2n}\right)^{\frac{2n+1}{2n}} \,d(\hat
x_{2n},t)\right]^{\frac{2n}{2n+1}}.
\end{align*}
The remaining task is to show that
\begin{equation}\label{eq:JGoal}
J \lesssim
\|f_1\|_{\frac{2n+1}{2}}\|f_{n+1}\|_{\frac{2n+1}{2}}\|f_n\|_{2n+1}\,
\prod_{j=2}^{n-1} \left(\|f_j\|_{2n+1}\|f_{n+j}\|_{2n+1}\right).
\end{equation}
We will next extract the $f_n$-term from the expression $J$.
First, by Minkowski's integral inequality, Fubini's theorem, and
the transformation $\tau \mapsto t= \tau+ x_n x_{2n}$, we obtain
the bound
\begin{align*}
J&\leq \int_{\mathbb{R}}
\left[\int_{\mathbb{R}^{2n-1}}\int_{\mathbb{R}} f_n(\hat
x_n,t)^{\frac{2n+1}{2n}}\prod_{\substack{j=1\\j\neq
n,2n}}^{2n}f_j(\pi_j(x,t-\tfrac{1}{2}x_n
x_{2n}))^{\frac{2n+1}{2n}}\,dt\,d\hat
x_{2n}\right]^{\frac{2n}{2n+1}}\,dx_{2n}.
\end{align*}
After this transformation, the $f_n$-term is independent of the
$n$-th coordinate of $x$. We can separate it from the other
factors by applying H\"older's inequality with exponents $p=2n$
and $p'= 2n/(2n-1)$ to the expression inside the square brackets.
This yields
\begin{align}\label{eq:J}
J&\leq \int_{\mathbb{R}} F_n\, F_{\Pi}\,dx_{2n}
\end{align}
where
\begin{displaymath}
F_n:= \left[\int_{\mathbb{R}^{2n-1}} f_n(\hat x_n,t)^{2n+1}
\,d(\hat x_{n,2n},t)\right]^{\frac{1}{2n+1}}
\end{displaymath}
and
\begin{displaymath}
F_{\Pi}:= \left[ \int_{\mathbb{R}^{2n-1}} \left(\int_{\mathbb{R}}
\prod_{\substack{j=1\\j\neq
n,2n}}^{2n}f_j(\pi_j(x,t-\tfrac{1}{2}x_n
x_{2n}))^{\frac{2n+1}{2n}} \,dx_n\right)^{\frac{2n}{2n-1}}
\,d(\hat x_{n,2n},t) \right]^{\frac{2n-1}{2n+1}}.
\end{displaymath}
Applying once more H\"older's inequality, but now to the
$x_{2n}$-integral in \eqref{eq:J}, and with exponents $p=2n+1$ and
$p'=(2n+1)/2n$, yields
\begin{displaymath}
J\leq \left( \int_{\mathbb{R}}
F_n^{2n+1}\,dx_{2n}\right)^{\frac{1}{2n+1}}\, \left(
\int_{\mathbb{R}}
F_{\Pi}^{\frac{2n+1}{2n}}\,dx_{2n}\right)^{\frac{2n}{2n+1}}= J_n
\cdot J_{\Pi}.
\end{displaymath}
Here
\begin{displaymath}
J_n := \left( \int_{\mathbb{R}}
F_n^{2n+1}\,dx_{2n}\right)^{\frac{1}{2n+1}} = \left(
\int_{\mathbb{R}^{2n}}  f_n(\hat x_n,t)^{2n+1} \,d(\hat
x_{n},t)\right)^{\frac{1}{2n+1}} =\|f_n\|_{2n+1}
\end{displaymath}
is one of the factors in the desired upper bound for $J$, recall
\eqref{eq:JGoal}. Hence, in order to prove \eqref{eq:JGoal}, it
suffices to show that
\begin{equation}\label{eq:JProdGoal}
J_{\Pi}:=  \left( \int_{\mathbb{R}}
F_{\Pi}^{\frac{2n+1}{2n}}\,dx_{2n}\right)^{\frac{2n}{2n+1}}
\lesssim \|f_1\|_{\frac{2n+1}{2}}\|f_{n+1}\|_{\frac{2n+1}{2}}
\prod_{j=2}^{n-1}\left( \|f_j\|_{2n+1}\,
\|f_{n+j}\|_{2n+1}\right).
\end{equation}
To do so, we will finally use our induction hypothesis. We start
by expanding
\begin{align*}
J_{\Pi}=  \left( \int_{\mathbb{R}}
 \left[ \int_{\mathbb{R}^{2n-1}} \left(\int_{\mathbb{R}}
\prod_{\substack{j=1\\j\neq
n,2n}}^{2n}f_j(\pi_j(x,t-\tfrac{1}{2}x_n
x_{2n}))^{\frac{2n+1}{2n}} \,dx_n\right)^{\frac{2n}{2n-1}}
\,d(\hat x_{n,2n},t)
\right]^{\frac{2n-1}{2n}}\,dx_{2n}\right)^{\frac{2n}{2n+1}}.
\end{align*}
Applying Minkowski's integral inequality inside the square
brackets, then Fubini's theorem and the transformation $t\mapsto
\tau=t - \frac{1}{2}x_n x_{2n}$ yields
\begin{align}\label{eq:J_Pi_Int}
J_{\Pi} &\leq \left(\int_{\mathbb{R}^2}
\left[\int_{\mathbb{R}^{2n-1}}\prod_{\substack{j=1\\j\neq
n,2n}}^{2n}f_j(\pi_j(x,\tau))^{\frac{2n+1}{2n-1}}\,d(\hat
x_{n,2n},\tau)\right]^{\frac{2n-1}{2n}}\,
d(x_n,x_{2n})\right)^{\frac{2n}{2n+1}}.
\end{align}
We recall that
\begin{equation}\label{eq:ProjFjForm}
f_j(\pi_j(x,\tau)) = \left\{\begin{array}{ll}f_j(\hat
x_j,\tau+\frac{1}{2}x_j x_{n+j}),&\text{if }j=1,\ldots,n-1,\\
f_j(\hat x_j,\tau-\frac{1}{2}x_{j-n} x_{j}),&\text{if
}j=n+1,\ldots,2n-1.
\end{array} \right.
\end{equation}
We will continue the upper bound for $J_{\Pi}$ by applying the
induction hypothesis to the expression inside the square brackets.
To do so, we temporarily denote points in $\mathbb{H}^{n-1}$ in
coordinates by $(u,t)=(u_1,\ldots,u_{2n-2},\tau)$. Here, $u$ is a
point in $\mathbb{R}^{2n-2}$, and similarly as before, $\hat u_k$
denotes the point in $\mathbb{R}^{2n-3}$ that is obtained from $u$
by deleting the $k$-th coordinate.

To write the inner integral on the right-hand side of
\eqref{eq:J_Pi_Int} in a form where the induction hypothesis is
applicable, we fix $x_n,x_{2n}\in \mathbb{R}$ and define the
 functions $g_{x_n,x_{2n},j}$, $j\in \{1,\ldots,2n-2\}$ on $\mathbb{R}^{2n-2}$:
\begin{equation}\label{eq:def_gj}
g_{x_n,x_{2n},j}(\hat u_j,t):=
\left\{\begin{array}{ll}f_1(u_2,\ldots,u_{n-1},x_n,u_{n},\ldots,u_{2n-2},x_{2n},t)^{\frac{2n+1}{2n-1}},&
j=1,\\
f_j(u_1,\ldots,u_{j-1},u_{j+1},\ldots,u_{n-1},x_n,u_{n},\ldots,u_{2n-2},x_{2n},t&)^{\frac{2n+1}{2n-1}},\\
&2 \leq j\leq n-2\\
f_{n-1}(u_1,\ldots,u_{n-2},x_n,u_{n},\ldots,u_{2n-2},x_{2n},t)^{\frac{2n+1}{2n-1}},&
 j=n-1,
\end{array}\right.
\end{equation}
and
\begin{equation}\label{eq:def_gjn}
g_{x_n,x_{2n},j}(\hat u_j,t):=\left\{\begin{array}{ll}
f_{n+1}(u_1,\ldots,u_{n-1},x_n,u_{n+1},\ldots,
u_{2n-2},x_{2n},t)^{\frac{2n+1}{2n-1}},&j=n,\\
f_{j+1}(u_1,\ldots,u_{n-1},x_n,u_{n},\ldots,u_{j-1},u_{j+1},\ldots
u_{2n-2},&x_{2n},t)^{\frac{2n+1}{2n-1}},\\&n+1\leq j\leq 2n-3,\\
f_{2n-1}(u_1,\ldots,u_{n-1},x_n,u_{n},\ldots,
u_{2n-3},x_{2n},t)^{\frac{2n+1}{2n-1}},&j=2n-2.
\end{array}\right.
\end{equation}
 With this notation in place, and
recalling \eqref{eq:ProjFjForm},
we can restate \eqref{eq:J_Pi_Int} equivalently as follows
\begin{align*}
J_{\Pi} &\leq \left(\int_{\mathbb{R}^2}
\left[\int_{\mathbb{R}^{2n-1}}\prod_{j=1}^{2n-2}g_{x_n,x_{2n},j}(\pi_j(u,t))\,d(u,t)\right]^{\frac{2n-1}{2n}}\,
d(x_n,x_{2n})\right)^{\frac{2n}{2n+1}},
\end{align*}
where $\pi_j$ now denotes the Heisenberg projection from
$\mathbb{H}^{n-1}$ to the vertical plane $\{u_j=0\}$ (identified
with $\mathbb{R}^{2n-2}$). The induction hypothesis applied to the
inner integral yields
\begin{align}\label{eq:J_Pi_BeforeHol}
J_{\Pi} &\lesssim \left(\int_{\mathbb{R}^2}
\left[\|g_{x_n,x_{2n},1}\|_{{\frac{2n-1}{2}}}\|g_{x_n,x_{2n},n}\|_{{\frac{2n-1}{2}}}\prod_{\substack{j=1\\j\notin{1,n}}}^{2n-2}
\|g_{x_n,x_{2n},j}\|_{{2n-1}}\right]^{\frac{2n-1}{2n}}\,
d(x_n,x_{2n})\right)^{\frac{2n}{2n+1}}.
\end{align}
Next we apply the multilinear H\"older inequality  with exponents
\begin{displaymath}
p_1=p_n=n\quad\text{and}\quad
p_2=\ldots=p_{n-1}=p_{n+1}=\ldots=p_{2n-2}=2n.
\end{displaymath}
Note that
\begin{displaymath}
\sum_{j=1}^{2n-2} \frac{1}{p_j} = \frac{2}{n}+ \frac{2n-4}{2n} =
1,
\end{displaymath}
as desired. Hence we deduce from \eqref{eq:J_Pi_BeforeHol} that
\begin{align*}
J_{\Pi} \lesssim & \left(\int_{\mathbb{R}^2}
\|g_{x_n,x_{2n},1}\|_{\frac{2n-1}{2}}^{\frac{2n-1}{2}}d(x_n,x_{2n})\right)^{\frac{2}{2n+1}}
\left(\int_{\mathbb{R}^2}
\|g_{x_n,x_{2n},n}\|_{\frac{2n-1}{2}}^{\frac{2n-1}{2}}d(x_n,x_{2n})\right)^{\frac{2}{2n+1}}
\cdot \\& \prod_{\substack{j=1\\j\notin{1,n}}}^{2n-2}
\left(\int_{\mathbb{R}^2}
\|g_{x_n,x_{2n},j}\|_{2n-1}^{2n-1}d(x_n,x_{2n})\right)^{\frac{1}{2n+1}}
.
\end{align*}
Recalling the definition of $g_{x_n,x_{2n},j}$ for
$j=1,\ldots,2n-2$ as stated in \eqref{eq:def_gj} and
\eqref{eq:def_gjn}, we obtain immediately
\begin{displaymath}
J_{\Pi} \lesssim
\|f_1\|_{\frac{2n+1}{2}}\|f_{n+1}\|_{\frac{2n+1}{2}}
\prod_{j=2}^{n-1}\left( \|f_j\|_{2n+1}\,
\|f_{n+j}\|_{2n+1}\right).
\end{displaymath}
as desired; recall \eqref{eq:JProdGoal}. This proves
\eqref{eq:JGoal} and thus establishes the statement about $k=1$ in
the induction claim \eqref{eq:vertex_eq} for $n$. The other values
of $k$ are treated analogously, and hence we have established
\eqref{eq:vertex_eq}.
\end{proof}

\section{Applications of the Loomis-Whitney inequalities in Heisenberg
groups}\label{s:Appl}

In this section, we derive the Gagliardo-Nirenberg-Sobolev
inequality in $\mathbb{H}^n$, and its variant Theorem
\ref{mainIntro3}, from the Loomis-Whitney inequality, Theorem
\ref{mainIntro2}. As a corollary of Theorem \ref{mainIntro3}, we
obtain the isoperimetric inequality in $\He^n$ (with a non-optimal
constant). At the end of the section, we also show how the
Loomis-Whitney inequality can be used, directly, to infer a
variant of the isoperimetric inequality, without passing through
the Sobolev inequality.

The arguments presented here are very standard
(\cite{MR0031538,MR102740,MR1863692}), and we claim no
originality. A version of this section, in the context of the
first Heisenberg group, was already contained in our joint work
\cite{fassler2020planar} with Tuomas Orponen. In his thesis
\cite{Bramati}, Bramati also gave an argument to deduce the
Gagliardo-Nirenberg-Sobolev and isoperimetric inequalities in
$\mathbb{H}^1$ from the strong version of the Loomis-Whitney
inequality stated in Theorem \ref{mainStrong}.

 We start by
recalling the statement of Theorem \ref{mainIntro3}:

\begin{thm}\label{t:sobolev} Let $f \in BV(\He^n)$. Then,
\begin{equation}\label{eq:GNS} \|f\|_{\frac{2n+2}{2n+1}} \lesssim \prod_{j=1}^{2n} \|X_jf\|^{\frac{1}{2n}}. \end{equation}
\end{thm}
Recall that $f \in BV(\He^n)$ if $f \in L^{1}(\He^n)$, and the
distributional derivatives $X_jf$, $j=1,\ldots,2n$, are finite
signed Radon measures. Smooth compactly supported functions are
dense in $BV(\He^n)$ in the sense that if $f \in BV(\He^n)$, then
there exists a sequence $\{\varphi_{k}\}_{k \in \N} \subset
C^{\infty}_{c}(\R^{2n+1})$ such that $\varphi_{k} \to f$ almost
everywhere (and in $L^{1}(\He^n)$ if desired), and
$\|Z\varphi_{k}\| \to \|Z f\|$ for $Z \in \{X_1,\ldots,X_{2n}\}$.
For a reference, see \cite[Theorem 2.2.2]{MR1437714}. With this
approximation in hand, it suffices to prove Theorem
\ref{t:sobolev} for, say, $f \in C^{1}_{c}(\R^{2n+1})$. The
following lemma contains most of the proof:
\begin{lemma}\label{l:sobolev} Let $f \in C^{1}_{c}(\R^{2n+1})$, and write
\begin{equation}\label{eq:Fk} F_{k} := \{p \in \R^{2n+1} : 2^{k - 1} \leq |f(p)| \leq 2^{k}\}, \qquad k \in \Z. \end{equation}
Then,
\begin{equation}\label{form20} |\pi_{j}(F_{k})| \leq 2^{-k + 2} \int_{F_{k - 1}} |X_jf|
,\quad j=1,\ldots 2n.\end{equation}
\end{lemma}
\begin{proof} By symmetry, it suffices to prove the inequality in \eqref{form20} for $j=1,\ldots,n$.
Let $w = (\hat x_{j},t) \in \pi_{j}(F_{k})$, denote by $e_j$ the
$j$-th unit vector, and fix $p = w \cdot x_j e_j \in F_{k}$ such
that $\pi_{j}(p) = w$. In particular, $|f(p)| \geq 2^{k - 1}$.
Recall the notation $\mathbb{L}_{j} = \mathrm{span}(e_j)=\{x_j e_j
: x_{j} \in \R\}$ and the definition of $\hat x_j$ given below
\eqref{eq:GroupProd}. Since $f$ is compactly supported, we may
pick another point $p' \in w \cdot \mathbb{L}_{j}$ such that
$f(p') = 0$. Since $|f|$ is continuous, we infer that there is a
non-degenerate line segment $I$ on the line $w \cdot
\mathbb{L}_{j}$ such that $2^{k - 2} \leq |f(q)| \leq 2^{k - 1}$
for all $q \in I$ (hence $I \subset F_{k - 1}$), and $|f|$ takes
the values $2^{k - 2}$ and $2^{k - 1}$, respectively, at the
endpoints $q_{i} = w \cdot x_{j,i}e_j$ of $I$, $i \in \{1,2\}$.
Define $\gamma(x_j) := w \cdot x_j e_j = (x,t - \tfrac{1}{2}x_j
x_{n+j})$. With this notation,
\begin{align*}
 2^{k - 2}& \leq |f(q_{1}) - f(q_{2})| \leq \int_{x_{j,1}}^{x_{j,2}} |(f \circ \gamma)'(x_j)| \, dx_j
 \\&\leq \int_{\{x_j : (x,t - \frac{1}{2}x_jx_{n+j})
 \in F_{k - 1}\}} |X_jf(x,t - \tfrac{1}{2}x_jx_{n+j})| \, dy.
  \end{align*}
Writing $\Phi(x,t) := (\hat x_{j},t) \cdot x_j e_j = (x,t -
\tfrac{1}{2}x_j x_{n+j})$, and integrating over
$$(x_1,\ldots,x_{j-1},x_{j+1},\ldots,x_{2n},t) = (\hat
x_{j},t) \in \pi_{j}(F_{k}) \subset \W_{j},$$ it follows that
\begin{equation}\label{form21}
2^{k - 2}|\pi_{j}(F_{k})|\leq \int_{\pi_{j}(F_{k})} \left[
\int_{\{x_{j} : \Phi(x,t) \in F_{k - 1}\}} |X_{j}f(\Phi(x,t))| \,
dx_{j} \right] \, d\hat{x}_{j} \, dt.
\end{equation}
Finally, we note that $J_{\Phi} = \mathrm{det\,} D\Phi \equiv 1$.
Therefore, using Fubini's theorem, and performing a change of
variables to the right-hand side of \eqref{form21}, we see that
\begin{align*} 2^{k - 2}|\pi_{j}(F_{k})|
& \leq \int_{\{(x,t) \in \R^{2n+1} : \Phi(x,t) \in F_{k - 1}\}} |X_{j}f(\Phi(x,t))| \, d(x,t)\\
& = \int_{F_{k - 1}} |X_{j}f(x,t)| \, d(x,t).
\end{align*} This completes the proof. \end{proof} We are then
prepared to prove Theorem \ref{t:sobolev}:
\begin{proof}[Proof of Theorem \ref{t:sobolev}]
Fix $f \in C^{1}_{c}(\R^{2n+1})$, and define the sets $F_{k}$, $k
\in \Z$, as in \eqref{eq:Fk}. Using first Theorem
\ref{mainIntro2}, then Lemma \ref{l:sobolev}, then the generalized
H\"older's inequality with $p_1=\ldots=p_{2n}=2n$, and finally the
embedding $\ell^{1} \hookrightarrow \ell^{(2n+2)/(2n+1)}$, we
estimate as follows:
\begin{align*} \int |f|^{\frac{2n+2}{2n+1}} &\sim \sum_{k \in \Z}
2^{\frac{(2n+2)k}{2n+1}}|F_{k}|\\
& \lesssim \sum_{k \in \Z} 2^{\frac{(2n+2)k}{2n+1}}\prod_{j=1}^{2n} |\pi_{j}(F_{k})|^{\frac{n+1}{n(2n+1)}}\\
& \lesssim \sum_{k \in \Z} \prod_{j=1}^{2n}
\Big(\int_{F_{k - 1}} |X_jf| \Big)^{\frac{n+1}{n(2n+1)}}\\
& \lesssim \prod_{j=1}^{2n} \Big[ \sum_{k \in \Z} \Big( \int_{F_{k - 1}} |X_jf| \Big)^{\frac{2n+2}{2n+1}} \Big]^{\frac{1}{2n}}\\
& \lesssim \prod_{j=1}^{2n} \Big[\sum_{k \in \Z} \int_{F_{k - 1}}
|X_jf| \Big]^{\frac{2n+2}{2n(2n+1)}} \sim \prod_{j=1}^{2n}
\|X_jf\|_{1}^{\frac{2n+2}{2n(2n+1)}}.
\end{align*} Raising both sides to the power $(2n+1)/(2n+2)$
completes the proof. \end{proof}

We conclude the section by discussing isoperimetric inequalities.
A measurable set $E \subset \He^n$ has \emph{finite horizontal
perimeter} if $\chi_{E} \in BV(\He^n)$. Here $\chi_{E}$ is the
characteristic function of $E$. Note that our definition of
$BV(\He^n)$ implies, in particular, that $|E| < \infty$. We follow
common practice, and write $P_{\mathbb{H}}(E) := \|\nabla_{\He}
\chi_{E}\|$. For more information on sets of finite horizontal
perimeter, see \cite{FSSC}. Now, applying Theorem \ref{t:sobolev}
to $f = \chi_{E}$, we recover the following isoperimetric
inequality (with a non-optimal constant):
\begin{thm} There exists a constant $C>0$ such that
\begin{equation}\label{eq:isop}
|E|^{\frac{2n+1}{2n+2}}\leq C P_{\mathbb{H}}(E)
\end{equation}
for any measurable set $E \subset \He^n$ of finite horizontal
perimeter.
\end{thm}
For $n=1$, this is Pansu's isoperimetric inequality
\cite{MR676380}, which has later been generalized to $\mathbb{H}^n
$ and beyond \cite{MR1404326,MR1312686}.
 We remark that the \emph{a priori} assumption
$|E| < \infty$ is critical here; for example the theorem evidently
fails for $E = \He^n$, for which $|E| = \infty$ but
$\|\nabla_{\He} \chi_{E}\| = 0$. We conclude the paper by deducing
a weaker version of \eqref{eq:isop} (even) more directly from the
Loomis-Whitney inequality. Namely, we claim that
\begin{equation}\label{weakIP} |E|^{\frac{2n+1}{2n+2}}\leq C \calH^{2n+1}_{d}(\partial E) \end{equation}
for any bounded measurable set $E \subset \He^n$, where
$\calH^{2n+1}_{d}$ denotes the $2n+1$-dimensional Hausdorff
measure on $\He^n$ with respect to the Kor\'{a}nyi distance (or
the standard left-invariant sub-Riemannian metric).
 This inequality
is, in general, weaker than \eqref{eq:isop}: at least for open
sets $E \subset \He^n$, the property $\calH^{2n+1}_{d}(\partial E)
< \infty$ implies that $P_{\He}(E) < \infty$, and then $P_{\He}(E)
\lesssim \calH^{2n+1}_{d}(\partial E)$, see \cite[Theorem
4.18]{MR2836591}. However, if $E$ is a bounded open set with
$C^{1}$ boundary, then $\mathcal{H}^{2n+1}_{d}(\partial E) \sim
P_{\He}(E)$, see \cite[Corollary 7.7]{FSSC}.

To prove \eqref{weakIP}, we need the following auxiliary result,
see \cite[Lemma 3.4]{MR3992573} and \cite[Remark 4.7]{MR4127898}:

\begin{lemma} Let $n\in \mathbb{N}$. There exists a constant $C_n > 0$ such that the following holds.
Let $\W \subset \mathbb{H}^n$ be a vertical subgroup of
codimension $1$. Then,
\begin{equation}\label{per2}
|\pi_{\W}(A)|\leq C \mathcal{H}^{2n+1}_d(A), \qquad A \subset
\He^n.
\end{equation}
\end{lemma}
\begin{proof}[Proof of \eqref{weakIP}]
Let $E\subset\mathbb{H}$ be bounded and measurable. We first claim that
\begin{align}\label{uguproj}
&\pi_{j}(E)\subseteq \pi_{j}(\partial E), \quad j=1,\ldots, 2n.
\end{align}
Let $w\in \pi_{j}(E)$ and consider $\pi_{j}^{-1}\{w\} = w \cdot
\mathbb{L}_{j}$ where $\mathbb{L}_{y}=\mathrm{span}(e_j)$. By
definition there exists $x_{j,1}\in \mathbb{R}$ such that $w \cdot
x_{j,1} e_j\in E$ and since $E$ is bounded there also exists
$x_{j,2}\in\mathbb{R}$ such that $w \cdot  x_{j,2}\in \mathbb{H}^n
\, \setminus \, \overline{E}$. Since $w \cdot \mathbb{L}_{j}$ is
connected, there finally exists $x_{j,3}\in\mathbb{R}$ such that
$w \cdot x_{j,3} e_j \in \partial E$ which immediately implies
\eqref{uguproj}. Using Theorem \ref{mainIntro2} and
\eqref{uguproj}, we get
\[
|E| \lesssim \prod_{j=1}^{2n}|\pi_j(\partial
E)|^{\frac{n+1}{n(2n+1)}}.
\]
Now the isoperimetric inequality \eqref{weakIP} follows using Lemma \ref{per2}.
\end{proof}



\section{Generalized Loomis-Whitney
inequalities by induction}\label{s:Generalized}

The  approach described in Section \ref{s:Hn} can be used to prove
something a bit more general, namely we can replace the vertical
Heisenberg projections $\pi_1,\ldots,\pi_{2n}$ by projection-type
mappings of the form
\begin{equation}\label{eq:rho_j}
\rho_j:\mathbb{R}^{2n+1}\to\mathbb{R}^{2n},\quad \rho_j(x,t)=(
\hat x_j, t + h_j(x)),\quad j=1,\ldots,2n,
\end{equation}
for suitable $\mathcal{C}^1$ maps $h_j:\mathbb{R}^{2n}\to
\mathbb{R}$. The precise condition is stated in Definition
\ref{d:inductiveL32L3property} and it is tailored so that a
Loomis-Whitney-type inequality for $\rho_1,\ldots,\rho_{2n}$ can
be established based on the $L^{3/2}$-$L^3$ boundedness of a
linear operator in the plane, analogously as we did for
$\pi_1,\ldots,\pi_{2n}$ and the Radon transform in Sections
\ref{s:H1}-\ref{s:Hn}. By a simple change-of-variables, one can
generalize the setting even slightly further, see Remark
\ref{r:general_rhoj}.

For arbitrary $\mathcal{C}^1$ functions $h_j$, the mappings
$\rho_j$ defined in \eqref{eq:rho_j}  satisfy a condition
analogous to \eqref{eq:RankCondition} for $\pi_j$, which ensures
by the coarea formula that the preimage of a Lebesgue null set in
$\mathbb{R}^{2n}$ under $\rho_j$ is a Lebesgue null set in
$\mathbb{R}^{2n+1}$. More precisely, we have
\begin{align*}
\det (D\rho_j D\rho_j^t)=
\det\begin{pmatrix}1&&&\\&\ddots&&\nabla_{\hat
x_j}h\\&&1&\\&\nabla_{\hat x_j}h&&(1+|\nabla h|^2)\end{pmatrix} &=
1+ (\partial_{x_j} h)^2.
\end{align*}
By the reasoning below Theorem \ref{mainIntroStrong} it follows
that $f\circ \rho_j$ is Lebesgue measurable on $\mathbb{R}^{2n+1}$
if $f$ is Lebesgue measurable on $\mathbb{R}^{2n}$.

\begin{definition}[$L^{3/2}$-$L^3$
property] \label{d:inductiveL32L3property} We say that a family
$\{h_1,\ldots,h_{2n}\}$ of $\mathcal{C}^1$ functions $h_j$ on
$\mathbb{R}^{2n}$ has the \emph{$L^{3/2}$-$L^3$ property} if there
exists a constant $C<\infty$ such that the following holds for all
$k=1,\ldots,n$:

\begin{itemize}
\item If $n>1$, then for every $\hat x_{k,n+k}\in
\mathbb{R}^{2n-2}$, the operator $T_{k,\hat x_{k,n+k}}$, defined
by
\begin{displaymath}
T_{k,\hat x_{k,n+k}}f(x_k,t):= \int_{\mathbb{R}}
f(x_{n+k},t+h_k(x)-h_{n+k}(x))\, dx_{n+k},\quad f\in
\mathcal{C}^{\infty}_c(\mathbb{R}^2)
\end{displaymath}
satisfies
\begin{displaymath}
\|T_{k,\hat x_{k,n+k}}f\|_3\leq C \|f\|_{\frac{3}{2}},\quad f\in
\mathcal{C}^{\infty}_c(\mathbb{R}^2).
\end{displaymath}
Here, the coordinates of $\hat x_{k,n+k}\in \mathbb{R}^{2n-2}$ are
$x_i$, $i\in \{1,\ldots,2n\}\setminus \{k,n+k\}$, and
$x=(x_1,\ldots,x_k,\ldots,x_{n+k},\ldots,x_{2n})$. \item  If
$n=1$, then the operator $T_1=T$, defined by
\begin{displaymath}
Tf(x_1,t):=\int_{\mathbb{R}}
f(x_2,t+h_1(x_1,x_2)-h_2(x_1,x_2))dx_2,\quad
f\in\mathcal{C}^{\infty}_c(\mathbb{R}^2)
\end{displaymath}
satisfies
\begin{displaymath}
\|Tf\|_3\leq C \|f\|_{\frac{3}{2}},\quad f\in
\mathcal{C}^{\infty}_c(\mathbb{R}^2).
\end{displaymath}
\end{itemize}
\end{definition}

We next give examples of functions $\{h_1,\ldots,h_{2n}\}$ with
the properties stated in Definition
\ref{d:inductiveL32L3property}. Essentially, for $k=1,\ldots,n$,
we take $h_k$ and $h_{n+k}$ to be polynomials of second degree as
functions of $x_k$ and $x_{n+k}$ so that Theorem \ref{t:PlanarOp}
is applicable. This class of examples includes the functions
\begin{equation}\label{eq:h_j_std}
h_j(x)=\left\{\begin{array}{ll}\tfrac{1}{2}x_j
x_{n+j},&j=1,\ldots,n,\\ -\tfrac{1}{2}
x_{j-n}x_j,&j=n+1,\ldots,2n.\end{array}\right.
\end{equation}
associated to the standard Heisenberg vertical coordinate
projections $\rho_j=\pi_j$, $j=1,\ldots,2n$.

\begin{ex}\label{ex:operators}
Fix $n>1$, $b_j\in \mathbb{R}$ and $c_{j,a}\in
\mathcal{C}^1(\mathbb{R}^{2n-2})$ for $j=1,\ldots,2n$ and
multi-indices $a\in
\mathcal{A}:=\{(0,0),(1,0),(0,1),(2,0),(0,2)\}$. For
$k=1,\ldots,n$, we define
\begin{displaymath}
h_k(x):=  b_k\, x_k x_{n+k} + \sum_{a=(a_1,a_2)\in
\mathcal{A}}c_{k,a}(\hat x_{k,n+k}) x_k^{a_1} x_{n+k}^{a_2}
\end{displaymath}
and
\begin{displaymath}
h_{n+k}(x):= b_{n+k}\, x_k x_{n+k} +
\sum_{a=(a_1,a_2)\in\mathcal{A}}c_{n+k,a}(\hat x_{k,n+k})x_k^{a_1}
x_{n+k}^{a_2}.
\end{displaymath}
Then the operators appearing in Definition
\ref{d:inductiveL32L3property} are given by
\begin{displaymath}
T_{k,\hat x_{k,n+k}}f(x_k,t):= \int_{\mathbb{R}} f\left(x_{n+k},t+
H_{k,n+k}(x)\right) \, dx_{n+k},\quad f\in
\mathcal{C}^{\infty}_c(\mathbb{R}^2),
\end{displaymath}
where
\begin{align*}
H_{k,n+k}(x):= (b_k-b_{n+k})x_k
x_{n+k}+\sum_{a=(a_1,a_2)\in\mathcal{A}}\left[c_{k,a}(\hat
x_{k,n+k})-c_{n+k,a}(\hat x_{k,n+k})\right]x_k^{a_1}
x_{n+k}^{a_2}.
\end{align*}
If $ b_k-b_{n+k}\neq 0$ for $k=1,\ldots,n$, then
$\{h_1,\ldots,h_{2n}\}$ has the $L^{3/2}$-$L^3$ property by
Theorem \ref{t:PlanarOp} with constant $C \lesssim
(\min_{k=1,\ldots,n}|b_k-b_{n+k}|)^{-1/3}$. This is the case in
particular for $\{h_1,\ldots,h_{2n}\}$ as in \eqref{eq:h_j_std}.
Hence, Theorems \ref{t:induction theorem} and
\ref{mainIntroStrong} are special cases of
 Theorems \ref{t:induction theorem_nonperp} and
\ref{mainIntroStrong_nonperp} below.
\end{ex}

We claim no originality for Theorem \ref{t:PlanarOp} that was
applied in the previous example. It is an instance of much more
general results available in the literature. We merely explain
here how the statement follows from the $L^{3/2}$-$L^3$ improving
property of (i) the Radon transform and (ii) convolution with a
measure on a parabola. Even though (i) involves integration over
lines with different slopes, and (ii) concerns convolution with a
\emph{fixed} parabola, both operators fit in the same framework
\cite[p.606]{MR1969206}.

\begin{thm}\label{t:PlanarOp}
Let $\alpha,\beta,\gamma,\delta,\epsilon,\kappa \in \mathbb{R}$.
If $\beta\neq 0$, then the operator $S$, defined by
\begin{displaymath}
Sf(x,t)=\int_{\mathbb{R}} f(y,t+ \alpha y^2 + \beta xy + \gamma
x^2 + \delta x + \epsilon y + \kappa)\,dy,\quad f\in
\mathcal{C}^{\infty}_c(\mathbb{R}^2),
\end{displaymath}
satisfies
\begin{equation}\label{eq:boundS}
\|Sf\|_{3}\lesssim |\beta|^{-1/3} \|f\|_{\frac{3}{2}},\quad f\in
\mathcal{C}^{\infty}_c(\mathbb{R}^2).
\end{equation}
\end{thm}

\begin{proof}
We divide the proof in two cases: $\alpha=0$ and $\alpha\neq 0$.
In the first case, we apply the $L^{3/2}$-$L^3$ improving property
of the Radon transform \cite{MR667786} (in the form of Theorem
\ref{t:T_n=1_bdd}). In the second case, we reduce matters to the
$L^{3/2}$-$L^3$ improving property of the convolution operator
with a measure on a parabola \cite{MR0358443,MR1049762,MR1887100}.

First, if $\alpha=0$, then, for $f\in
\mathcal{C}^{\infty}_c(\mathbb{R}^2)$, we relate $Sf$ to the
operator $T$ from Theorem \ref{t:T_n=1_bdd} as follows:
\begin{align*}
Sf(x,t)= \int_{\mathbb{R}} f(y,t+ [\beta x+\epsilon]y + [\gamma
x^2 + \delta x + \kappa])\,dy= Tf(\beta x+ \epsilon,t+\gamma x^2
+\delta x +\kappa).
\end{align*}
Thus
\begin{align*}
\|Sf\|_{3}= \left(\int_{\mathbb{R}^2} |Tf(\beta x+
\epsilon,t+\gamma x^2 +\delta x
+\kappa)|^3\,d(x,t)\right)^{\frac{1}{3}}&=
|\beta|^{-1/3}\left(\int_{\mathbb{R}^2}
|Tf(\xi,\tau)|^3\,d(\xi,\tau)\right)^{\frac{1}{3}}\\& =
|\beta|^{-1/3} \|Tf\|_3,
\end{align*}
and hence Theorem \ref{t:T_n=1_bdd} implies \eqref{eq:boundS} in
that case.

If $\alpha\neq 0$, we instead reduce matters to \cite{MR0358443},
or the more general \cite[Theorem 1]{MR1049762}. A special case of
that theorem says that
\begin{equation}\label{eq:ConvParab}
\|\mu_{\alpha} \ast f\|_3 \lesssim \|f\|_{\frac{3}{2}},\quad f\in
L^{\frac{3}{2}}(\mathbb{R}^2),
\end{equation}
where
\begin{displaymath}
\mu_{\alpha} \ast f (x,t):= \int_{\mathbb{R}} f((x,t)-(y,\alpha
y^2))|\alpha|^{1/3}\,dy,
\end{displaymath}
see also \cite[Theorem 1]{MR1887100}. To employ this result, we
aim to relate $Sf$ for $f\in \mathcal{C}^{\infty}_c(\mathbb{R}^2)$
to $\mu_{\alpha} \ast f$. We apply elementary transformations to
one of the expressions that appear in the definition of $Sf$,
namely
\begin{align*}
t+ \alpha y^2 + \beta xy + \gamma x^2 &+ \delta x + \epsilon y +
\kappa
\\&= \alpha \left[y +
\tfrac{1}{2}\left(\tfrac{\beta}{\alpha}x+\tfrac{\epsilon}{\alpha}\right)\right]^2+
\left[ -\tfrac{\alpha}{4}\left(\tfrac{\beta}{\alpha}x+
\tfrac{\epsilon}{\alpha}\right)^2 +\gamma x^2+ \delta x + \kappa +
t\right].
\end{align*}
Hence, by the change-of-variables $y\mapsto \eta = - [y +
\tfrac{1}{2}\left(\tfrac{\beta}{\alpha}x+\tfrac{\epsilon}{\alpha}\right)]$,
we obtain
\begin{align*}
Sf(x,t)= &\int_{\mathbb{R}} f\left(y, \alpha \left[y +
\tfrac{1}{2}\left(\tfrac{\beta}{\alpha}x+\tfrac{\epsilon}{\alpha}\right)\right]^2+\left[
-\tfrac{\alpha}{4}\left(\tfrac{\beta}{\alpha}x+
\tfrac{\epsilon}{\alpha}\right)^2 + \gamma x^2+\delta x + \kappa +
t\right]\right)\, dy\\
=&\int_{\mathbb{R}} f\left( -
\tfrac{1}{2}\left(\tfrac{\beta}{\alpha}x+\tfrac{\epsilon}{\alpha}\right)-\eta,
\left[ -\tfrac{\alpha}{4}\left(\tfrac{\beta}{\alpha}x+
\tfrac{\epsilon}{\alpha}\right)^2 + \gamma x^2+\delta x + \kappa +
t\right]-(-\alpha)\eta^2\right)\, d\eta\\
=&|\alpha|^{-1/3} \mu_{-\alpha}\ast
f\left(\Phi(x,t)\right),\end{align*} with
\begin{displaymath}
\Phi(x,t):= \left(-
\tfrac{1}{2}\left(\tfrac{\beta}{\alpha}x+\tfrac{\epsilon}{\alpha}\right),
-\tfrac{\alpha}{4}\left(\tfrac{\beta}{\alpha}x+
\tfrac{\epsilon}{\alpha}\right)^2 + \gamma x^2 +\delta x + \kappa
+ t\right).
\end{displaymath}
Since
\begin{displaymath}
|\det D\Phi(x,t)|=\left|\beta\right| \,\left|2\alpha\right|^{-1},
\end{displaymath}
we find that
\begin{displaymath}
\|Sf\|_3 =|\alpha|^{-1/3} \|\left(\mu_{-\alpha}\ast f \right)\circ
\Phi\|_3 =|\alpha|^{-1/3} \left|\beta\right|^{-1/3}
\,\left|2\alpha\right|^{1/3} \|\mu_{-\alpha}\ast f \|_3.
\end{displaymath}
Thus \eqref{eq:boundS} in the case $\alpha\neq 0$ follows from
\eqref{eq:ConvParab}.
\end{proof}

We next prove a generalization of Theorem \ref{t:induction
theorem} that applies in particular to mappings
$\rho_1,\ldots,\rho_{2n}$ as in \eqref{eq:rho_j} for
$h_1,\ldots,h_{2n}$ as in Example \ref{ex:operators}.

\begin{thm}\label{t:induction theorem_nonperp}
Let $n\in \mathbb{N}$. Assume that $\{h_1,\ldots,h_{2n}\}$ is a
family of $\mathcal{C}^1$ functions on $\mathbb{R}^{2n}$ with the
$L^{3/2}$-$L^3$ property and define
\begin{displaymath}
\rho_j:\mathbb{R}^{2n+1}\to\mathbb{R}^{2n},\quad
\rho_j(x,t)=\left(\hat x_j, t + h_j(x)\right),\quad j=1,\ldots,2n.
\end{displaymath}
Then, for all nonnegative Lebesgue measurable functions
$f_1,\ldots,f_{2n}$ on $\mathbb{R}^{2n}$, we have
\begin{equation}\label{eq:vertex_eq_nonperp}
\int_{\mathbb{R}^{2n+1}} \prod_{j=1}^{2n} f_j(\rho_j(p))\;dp
\lesssim \|f_k\|_{\frac{2n+1}{2}}\|f_{n+k}\|_{\frac{2n+1}{2}}
\prod_{\substack{j=1\\j\neq k}}^n\left(
\|f_j\|_{2n+1}\,\|f_{n+j}\|_{2n+1}\right),\quad k\in
\{1,\ldots,n\},
\end{equation}
with an implicit constant that may depend on $n$ and the
boundedness constant $C$ associated to the family
$\{h_1,\ldots,h_{2n}\}$. If $n=1$, then
\eqref{eq:vertex_eq_nonperp} reads
\begin{displaymath}
\int_{\mathbb{R}^{3}} f_1(\rho_1(p))f_2(\rho_2(p))\;dp \lesssim
\|f_1\|_{\frac{3}{2}}\|f_{2}\|_{\frac{3}{2}}.
\end{displaymath}
\end{thm}

The statement can be deduced by following the proof of Theorem
 \ref{t:induction theorem} almost verbatim. We decided to
give the argument for Theorem
 \ref{t:induction theorem}  first in Section \ref{s:Hn} since it is a bit easier to read
 and helps motivate the more general discussion in the present
 section. Below we merely explain how to adapt the proof of Theorem
 \ref{t:induction theorem} to establish Theorem  \ref{t:induction
 theorem_nonperp}.

\begin{proof} It suffices to verify the claim for
nonnegative, smooth, and compactly supported functions
$f_1,\ldots,f_{2n}$. The case $n=1$ follows directly from the
$L^{3/2}$-$L^3$ property of $\{h_1,h_2\}$ in Definition
\ref{d:inductiveL32L3property}, and a simple change-of-variables
argument, observing that
\begin{align*}
\int_{\mathbb{R}^3} f_1(\rho_1(p))f_2(\rho_2(p))&\,dp =
\int_{\mathbb{R}^3} f_1(x_2,t + h_1(x_1,x_2))
f_2(x_1,t+h_2(x_1,x_2))\,d(x_1,x_2,t)\\
&= \int_{\mathbb{R}^2}
f_2(x_1,\tau)\left(\int_{\mathbb{R}}f_1(x_2,\tau+h_1(x_1,x_2)-h_2(x_1,x_2))\,dx_2\right)\,d(x_1,\tau)\\
&= \int_{\mathbb{R}^2}
f_2(x_1,\tau)T_1f_1(x_1,\tau)\,d(x_1,\tau)\\&\leq \|T_1f_1\|_3
\|f_2\|_{\frac{3}{2}}\leq C \|f_1\|_{\frac{3}{2}}
\|f_2\|_{\frac{3}{2}},
\end{align*}
for nonnegative $f_1,f_2\in \mathcal{C}^{\infty}_c(\mathbb{R}^2)$.

Suppose next that the statement of Theorem \ref{t:induction
theorem_nonperp}  has already been established for all natural
numbers up to $n-1$. We will argue that it holds also for the
integer $n$. To this end, we fix an arbitrary family
$\{h_1,\ldots,h_{2n}\}$ of $\mathcal{C}^1$ functions
$\mathbb{R}^{2n}\to \mathbb{R}$ with the $L^{3/2}$-$L^3$ property.
Given nonnegative $\mathcal{C}^{\infty}_{c}$ functions
$f_1,\ldots,f_{2n}$, we aim to show the $n$ inequalities stated in
\eqref{eq:vertex_eq_nonperp}, and by symmetry it suffices to
discuss this for $k=1$. By the same argument as in the proof of
Theorem \ref{t:induction theorem}, but now using the
transformation $t\mapsto t+ h_{2n}(x) = \tau$, we find that
\begin{align}\label{eq:I_ast}
I:= &\int_{\mathbb{R}^{2n}} \int_{\mathbb{R}} \prod_{j=1}^{2n}
f_j(\rho_j(x,t))\,dt\,dx\\
= & \int_{\mathbb{R}^{2n}}  f_{2n}(\hat x_{2n}, \tau)\left[
\int_{\mathbb{R}}  f_n(\hat x_n,\tau +h_n(x)-h_{2n}(x))
\prod_{\substack{j=1\\j\neq
n,2n}}^{2n}f_j(\rho_j(x,\tau-h_{2n}(x)))\,dx_{2n}\right]\,d(\hat
x_{2n},t).\notag
\end{align}
Applying H\"older's inequality, we can split off the factor with
$f_{2n}$ (which no longer depends on $x_{2n}$) and we obtain
$I\leq \|f_{2n}\|_{2n+1}\, J$ with
\begin{align*}
J:=\left[\int_{\mathbb{R}^{2n}} \left(\int_{\mathbb{R}}f_n(\hat
x_n,\tau +h_n(x)-h_{2n}(x))\prod_{\substack{j=1\\j\neq
n,2n}}^{2n}f_j(\rho_j(x,\tau-h_{2n}(x)))\,d
x_{2n}\right)^{\frac{2n+1}{2n}} \,d(\hat
x_{2n},t)\right]^{\frac{2n}{2n+1}}.
\end{align*}
The remaining task is to show that
\begin{equation}\label{eq:JGoal_nonperp}
J \lesssim_{n,C}
\|f_1\|_{\frac{2n+1}{2}}\|f_{n+1}\|_{\frac{2n+1}{2}}\|f_n\|_{2n+1}\,
\prod_{j=2}^{n-1} \left(\|f_j\|_{2n+1}\|f_{n+j}\|_{2n+1}\right),
\end{equation}
and this is done as in the proof of Theorem \ref{t:induction
theorem}, but using the transformation $\tau \mapsto t= \tau+
h_n(x)-h_{2n}(x)$. Then, as in the proof of Theorem
\ref{t:induction theorem}, we find that in order to prove
\eqref{eq:JGoal_nonperp}, it suffices to show that
\begin{equation}\label{eq:JProdGoal_nonperp}
J_{\Pi}\lesssim_{n,C}
\|f_1\|_{\frac{2n+1}{2}}\|f_{n+1}\|_{\frac{2n+1}{2}}
\prod_{j=2}^{n-1}\left( \|f_j\|_{2n+1}\,
\|f_{n+j}\|_{2n+1}\right),
\end{equation}
where
\begin{align*}
J_{\Pi}:=  \left( \int_{\mathbb{R}}
 \left[ \int_{\mathbb{R}^{2n-1}} \left(\int_{\mathbb{R}}
\prod_{\substack{j=1\\j\neq
n,2n}}^{2n}f_j(\rho_j(x,t-h_n(x)))^{\frac{2n+1}{2n}}
\,dx_n\right)^{\frac{2n}{2n-1}} \,d(\hat x_{n,2n},t)
\right]^{\frac{2n-1}{2n}}\,dx_{2n}\right)^{\frac{2n}{2n+1}}.
\end{align*}
Applying Minkowski's integral inequality inside the square
brackets, then Fubini's theorem and the transformation $t\mapsto
\tau=t - h_n(x)$ yields
\begin{align}\label{eq:J_Pi_Int_nonperp}
J_{\Pi} &\leq \left(\int_{\mathbb{R}^2}
\left[\int_{\mathbb{R}^{2n-1}}\prod_{\substack{j=1\\j\neq
n,2n}}^{2n}f_j(\rho_j(x,\tau))^{\frac{2n+1}{2n-1}}\,d(\hat
x_{n,2n},\tau)\right]^{\frac{2n-1}{2n}}\,
d(x_n,x_{2n})\right)^{\frac{2n}{2n+1}}.
\end{align}
We recall that
\begin{equation}\label{eq:ProjFjForm_nonperp}
f_j(\rho_j(x,\tau)) = f_j(\hat x_j, \tau + h_j(x)).
\end{equation}
We will continue the upper bound for $J_{\Pi}$ by applying the
induction hypothesis to the expression inside the square brackets.
To do so, we temporarily denote points in $\mathbb{H}^{n-1}$ in
coordinates by $(u,t)=(u_1,\ldots,u_{2n-2},\tau)$. Here, $u$ is a
point in $\mathbb{R}^{2n-2}$, and similarly as before, $\hat u_k$
denotes the point in $\mathbb{R}^{2n-3}$ that is obtained from $u$
by deleting the $k$-th coordinate. With this notation in place,
and recalling \eqref{eq:ProjFjForm_nonperp}, we can restate
\eqref{eq:J_Pi_Int_nonperp} equivalently as follows
\begin{align*}
J_{\Pi} &\leq \left(\int_{\mathbb{R}^2}
\left[\int_{\mathbb{R}^{2n-1}}\prod_{j=1}^{2n-2}g_{x_n,x_{2n},j}(\widetilde{\rho}_{j,x_n,x_{2n}}(u,t))\,d(u,t)\right]^{\frac{2n-1}{2n}}\,
d(x_n,x_{2n})\right)^{\frac{2n}{2n+1}},
\end{align*}
where $g_{x_n,x_{2n},j}(\hat u_j,t)$ are defined exactly as in
\eqref{eq:def_gj}-\eqref{eq:def_gjn}
and
\begin{displaymath}
\widetilde{\rho}_{j,x_n,x_{2n}}(u,t)=\left\{\begin{array}{ll}\left(\hat
u_j, t +
h_j(u_1,\ldots,u_{n-1},x_n,u_{n},\ldots,u_{2n-2},x_{2n})\right),&1\leq
j\leq n-1,\\ \left(\hat u_j, t +
h_{j+1}(u_1,\ldots,u_{n-1},x_n,u_{n},\ldots
u_{2n-2},x_{2n})\right),&n\leq j\leq 2n-2.\end{array}\right.
\end{displaymath}
Thus, the functions $\widetilde{\rho}_{j,x_n,x_{2n}}$ are as in
the statement of Theorem \ref{t:induction theorem_nonperp} for
$n-1$, with
\begin{displaymath}
\widetilde{h}_{j,x_n,x_{2n}}(u):= \left\{\begin{array}{ll}
h_j(u_1,\ldots,u_{n-1},x_n,u_{n},\ldots,u_{2n-2},x_{2n}),&
1\leq j\leq n-1,\\
h_{j+1}(u_1,\ldots,u_{n-1},x_n,u_{n},\ldots,u_{2n-2},x_{2n}),&n\leq
j\leq 2n-2.\end{array}\right.
\end{displaymath}
 In particular, if $\{h_1,\ldots,h_{2n}\}$ has the
$L^{3/2}$-$L^3$
 property with constant $C$ as assumed, then so does
$\{\widetilde{h}_{1,x_n,x_{2n}},\ldots,\widetilde{h}_{2n-2,x_n,x_{2n}}\}$
for every $(x_n,x_{2n})\in \mathbb{R}^2$. The induction hypothesis
applied to the inner integral therefore yields
\begin{align}\label{eq:J_Pi_BeforeHol_nonperp}
J_{\Pi} &\lesssim_C \left(\int_{\mathbb{R}^2}
\left[\|g_{x_n,x_{2n},1}\|_{{\frac{2n-1}{2}}}\|g_{x_n,x_{2n},n}\|_{{\frac{2n-1}{2}}}\prod_{\substack{j=1\\j\notin{1,n}}}^{2n-2}
\|g_{x_n,x_{2n},j}\|_{{2n-1}}\right]^{\frac{2n-1}{2n}}\,
d(x_n,x_{2n})\right)^{\frac{2n}{2n+1}}.
\end{align}
At the point, the proof can be concluded as in the case of Theorem
\ref{t:induction theorem}, recalling that the functions
$g_{x_n,x_{2n},j}$ have been defined exactly as in
\eqref{eq:def_gj}-\eqref{eq:def_gjn}.
\end{proof}

As in the case of the Heisenberg vertical coordinate projections,
we can use multilinear interpolation to deduce a Loomis-Whitney
type inequality for generalized projections
$\{\rho_1,\ldots,\rho_{2n}\}$.

\begin{thm}\label{mainIntroStrong_nonperp}
 Fix $n\in \mathbb{N}$, $n>1$. Given a family $\{h_1,\ldots,h_{2n}\}$ of
 $\mathcal{C}^1$ functions on $\mathbb{R}^{2n}$ that has the
 $L^{3/2}$-$L^3$ property with constant $C$, we define
\begin{displaymath}
\rho_j:\mathbb{R}^{2n+1}\to\mathbb{R}^{2n},\quad
\rho_j(x,t)=\left(\hat x_j, t + h_j(x)\right),\quad j=1,\ldots,2n.
\end{displaymath}Then
\begin{equation}\label{form18Strong_nonperp} \int_{\mathbb{R}^{2n+1}} \prod_{j=1}^{2n} f_j(\rho_j(p))\,dp \lesssim \prod_{j=1}^{2n}
\|f_j\|_{\frac{n(2n+1)}{n+1}},
\end{equation}
for all nonnegative Lebesgue measurable functions
$f_1,\ldots,f_{2n}$ on $\mathbb{R}^{2n}$, where the implicit
constant may depend on $n$  and $C$.
\end{thm}

\begin{remark}\label{r:general_rhoj}
A straightforward generalization of Theorem
\ref{mainIntroStrong_nonperp} can be obtained for the family
$\{\Phi_j \circ \rho_j :\; j=1,\ldots,2n\}$, where
$\Phi_j:\mathbb{R}^{2n}\to \mathbb{R}^{2n}$ are $\mathcal{C}^1$
diffeomorphisms with $\Lambda:= \min_{j=1,\ldots,2n}|\det
D\Phi_j|>0$ and $\rho_j$ are as in Theorem
\ref{mainIntroStrong_nonperp}. Indeed, simply apply Theorem
\ref{mainIntroStrong_nonperp} to the functions $g_j:= f_j \circ
\Phi_j$, $j=1,\ldots,2n$, and then perform changes-of-variables in
the integrals in $\|g_j\|_{\frac{n(2n+1)}{n+1}}$ to deduce that
\begin{displaymath} \int_{\mathbb{R}^{2n+1}} \prod_{j=1}^{2n} f_j(\Phi_j \circ \rho_j(p))\,dp \lesssim_{n,C,\Lambda} \prod_{j=1}^{2n}
\|f_j\|_{\frac{n(2n+1)}{n+1}}
\end{displaymath}
for all nonnegative Lebesgue measurable functions
$f_1,\ldots,f_{2n}$ on $\mathbb{R}^{2n}$.
\end{remark}

\begin{proof}[Proof of Theorem \ref{mainIntroStrong_nonperp} using Theorem \ref{t:induction theorem_nonperp}]
By the comment made at the beginning of the proof of Theorem
\ref{t:induction theorem_nonperp}, we already know the case $n=1$
of Theorem \ref{mainIntroStrong_nonperp}. Suppose that the
statement of Theorem \ref{t:induction theorem_nonperp} holds for a
given integer $n>1$. Fix mappings $h_j$ and $\rho_j$,
$j=1,\ldots,2n$, as in the statement of Theorems \ref{t:induction
theorem_nonperp} and \ref{mainIntroStrong_nonperp}.
 Our aim is
to verify \eqref{form18Strong_nonperp} for all nonnegative
measurable functions $f_1,\ldots,f_{2n}$ on $\mathbb{R}^{2n}$. The
desired inequality can be spelled out as follows:
\begin{equation}\label{eq:DesiredIneqExplicit_nonperp}
\int_{\mathbb{R}^{2n+1}} \prod_{j=1}^{2n} f_j(\hat
x_j,t+h_j(x))\;d(x,t) \lesssim \prod_{j=1}^{2n}
\|f_j\|_{\frac{n(2n+1)}{n+1}}.
\end{equation}
Similarly as in the proof of Theorem \ref{mainIntroStrong}, we
introduce a suitable multilinear operator $T$. Namely, for all
finitely simple functions $g_1,\ldots,g_{2n-1}$ on
$\mathbb{R}^{2n}$, we define
\begin{align*}\label{def:T_nonperp}
T (g_1,\ldots,g_{2n-1})(\hat x_{2n},\tau)
:=\int_{\mathbb{R}}g_n(\hat x_n,\tau+h_n(x)-h_{2n}(x))
\prod_{\substack{j=1\\j\neq n}}^{2n-1} g_j(\rho_j(x,
\tau-h_{2n}(x)))\mathrm{d}x_{2n}.\notag
\end{align*}
Hence, by the same computation that led to \eqref{eq:I_ast}, we
find for all finitely simple functions $f_1,\ldots,f_{2n-1}$ and
nonnegative measurable function $f_{2n}$ that
\begin{align*}
\int_{\mathbb{R}^{2n+1}} \prod_{j=1}^{2n} f_j(\rho_j(p))\;dp =
\int_{\mathbb{R}^{2n}} T(f_1,\ldots,f_{2n-1})(w) f_{2n}(w)\; dw.
\end{align*}
From this point on, the argument is entirely abstract and does no
longer use the specific form of the operator $T$. Analogously as
in the proof of Theorem \ref{mainIntroStrong}, the inequalities we
obtained in Theorem \ref{t:induction theorem_nonperp} yield bounds
of the form \eqref{eq:OpAssk_inproof} for the operator $T$. These
bounds can be combined using multilinear interpolation, as in
Theorem \ref{t:Grafakos}, to yield a bound of the form
\eqref{eq:GoalOperatorBound_inproof} for the operator $T$, which
eventually gives \eqref{eq:DesiredIneqExplicit_nonperp}.
\end{proof}

\bibliographystyle{plain}
\bibliography{references}

\def\cprime{$'$}
\begin{thebibliography}{10}

\bibitem{MR1863692}
Keith Ball.
\newblock Convex geometry and functional analysis.
\newblock In {\em Handbook of the geometry of {B}anach spaces, {V}ol. {I}},
  pages 161--194. North-Holland, Amsterdam, 2001.

\bibitem{MR3047423}
Zolt\'{a}n~M. Balogh, Estibalitz Durand-Cartagena, Katrin F\"{a}ssler, Pertti
  Mattila, and Jeremy~T. Tyson.
\newblock The effect of projections on dimension in the {H}eisenberg group.
\newblock {\em Rev. Mat. Iberoam.}, 29(2):381--432, 2013.

\bibitem{MR2955184}
Zolt\'{a}n~M. Balogh, Katrin F\"{a}ssler, Pertti Mattila, and Jeremy~T. Tyson.
\newblock Projection and slicing theorems in {H}eisenberg groups.
\newblock {\em Adv. Math.}, 231(2):569--604, 2012.

\bibitem{2018arXiv181111052B}
Jonathan Bennett, Neal Bez, Stefan Buschenhenke, Michael~G. Cowling, and
  Taryn~C. Flock.
\newblock On the nonlinear {B}rascamp-{L}ieb inequality.
\newblock {\em Duke Mathematical Journal}, 169(17):3291--3338, Nov 2020.

\bibitem{MR2377493}
Jonathan Bennett, Anthony Carbery, Michael Christ, and Terence Tao.
\newblock The {B}rascamp-{L}ieb inequalities: finiteness, structure and
  extremals.
\newblock {\em Geom. Funct. Anal.}, 17(5):1343--1415, 2008.

\bibitem{Benn}
Jonathan Bennett, Anthony Carbery, and James Wright.
\newblock A non-linear generalisation of the {L}oomis-{W}hitney inequality and
  applications.
\newblock {\em Math. Res. Lett.}, 12(4):443--457, 2005.

\bibitem{MR0482275}
J\"{o}ran Bergh and J\"{o}rgen L\"{o}fstr\"{o}m.
\newblock {\em Interpolation spaces. {A}n introduction}.
\newblock Springer-Verlag, Berlin-New York, 1976.
\newblock Grundlehren der Mathematischen Wissenschaften, No. 223.

\bibitem{Bob}
S.~G. Bobkov and F.~L. Nazarov.
\newblock On convex bodies and log-concave probability measures with
  unconditional basis.
\newblock In {\em Geometric aspects of functional analysis}, volume 1807 of
  {\em Lecture Notes in Math.}, pages 53--69. Springer, Berlin, 2003.

\bibitem{Bramati}
Roberto Bramati.
\newblock {\em Geometric integral inequalities on homogeneous spaces}.
\newblock PhD Thesis. Universit{\`a} degli Studi di Padova, 2019.

\bibitem{MR412366}
Herm~Jan Brascamp and Elliott~H. Lieb.
\newblock Best constants in {Y}oung's inequality, its converse, and its
  generalization to more than three functions.
\newblock {\em Advances in Math.}, 20(2):151--173, 1976.

\bibitem{MR167830}
A.-P. Calder\'{o}n.
\newblock Intermediate spaces and interpolation, the complex method.
\newblock {\em Studia Math.}, 24:113--190, 1964.

\bibitem{Salani}
Stefano Campi, Paolo Gronchi, and Paolo Salani.
\newblock A proof of a {L}oomis-{W}hitney type inequality via optimal
  transport.
\newblock {\em J. Math. Anal. Appl.}, 471(1-2):489--495, 2019.

\bibitem{MR2312336}
L.~Capogna, D.~Danielli, S.~D. Pauls, and J.~T. Tyson.
\newblock {\em An introduction to the {H}eisenberg group and the
  sub-{R}iemannian isoperimetric problem}, volume 259 of {\em Progress in
  Mathematics}.
\newblock Birkh\"auser Verlag, Basel, 2007.

\bibitem{MR1312686}
Luca Capogna, Donatella Danielli, and Nicola Garofalo.
\newblock The geometric {S}obolev embedding for vector fields and the
  isoperimetric inequality.
\newblock {\em Comm. Anal. Geom.}, 2(2):203--215, 1994.

\bibitem{MR3992573}
Vasileios Chousionis, Katrin F\"{a}ssler, and Tuomas Orponen.
\newblock Intrinsic {L}ipschitz graphs and vertical {$\beta$}-numbers in the
  {H}eisenberg group.
\newblock {\em Amer. J. Math.}, 141(4):1087--1147, 2019.

\bibitem{MR1654767}
Michael Christ.
\newblock Convolution, curvature, and combinatorics: a case study.
\newblock {\em Internat. Math. Res. Notices}, (19):1033--1048, 1998.

\bibitem{MR4176542}
Michael Christ, Spyridon Dendrinos, Betsy Stovall, and Brian Street.
\newblock Endpoint {L}ebesgue estimates for weighted averages on polynomial
  curves.
\newblock {\em Amer. J. Math.}, 142(6):1661--1731, 2020.

\bibitem{MR1945281}
Michael Christ and M.~Burak Erdogan.
\newblock Mixed norm estimates for a restricted {X}-ray transform.
\newblock volume~87, pages 187--198. 2002.
\newblock Dedicated to the memory of Thomas H. Wolff.

\bibitem{MR1726701}
Michael Christ, Alexander Nagel, Elias~M. Stein, and Stephen Wainger.
\newblock Singular and maximal {R}adon transforms: analysis and geometry.
\newblock {\em Ann. of Math. (2)}, 150(2):489--577, 1999.

\bibitem{MR1049762}
S.~W. Drury.
\newblock Degenerate curves and harmonic analysis.
\newblock {\em Math. Proc. Cambridge Philos. Soc.}, 108(1):89--96, 1990.

\bibitem{MR3495435}
Katrin F\"{a}ssler and Risto Hovila.
\newblock Improved {H}ausdorff dimension estimate for vertical projections in
  the {H}eisenberg group.
\newblock {\em Ann. Sc. Norm. Super. Pisa Cl. Sci. (5)}, 15:459--483, 2016.

\bibitem{fassler2020planar}
Katrin F{\"a}ssler, Tuomas Orponen, and Andrea Pinamonti.
\newblock Planar incidences and geometric inequalities in the {H}eisenberg
  group, 2020.
\newblock arXiv:2003.05862.

\bibitem{MR4127898}
Katrin F\"{a}ssler, Tuomas Orponen, and S\'{e}verine Rigot.
\newblock Semmes surfaces and intrinsic {L}ipschitz graphs in the {H}eisenberg
  group.
\newblock {\em Trans. Amer. Math. Soc.}, 373(8):5957--5996, 2020.

\bibitem{Finner}
Helmut Finner.
\newblock A generalization of {H}\"{o}lder's inequality and some probability
  inequalities.
\newblock {\em Ann. Probab.}, 20(4):1893--1901, 1992.

\bibitem{FSSC}
B.~Franchi, R.~Serapioni, and F.~Serra~Cassano.
\newblock Rectifiability and perimeter in the {H}eisenberg group.
\newblock {\em Math. Ann.}, 321(3):479--531, 2001.

\bibitem{MR1437714}
Bruno Franchi, Raul Serapioni, and Francesco Serra~Cassano.
\newblock Meyers-{S}errin type theorems and relaxation of variational integrals
  depending on vector fields.
\newblock {\em Houston J. Math.}, 22(4):859--890, 1996.

\bibitem{MR2836591}
Bruno Franchi, Raul Serapioni, and Francesco Serra~Cassano.
\newblock Differentiability of intrinsic {L}ipschitz functions within
  {H}eisenberg groups.
\newblock {\em J. Geom. Anal.}, 21(4):1044--1084, 2011.

\bibitem{MR102740}
Emilio Gagliardo.
\newblock Propriet\`a di alcune classi di funzioni in pi\`u variabili.
\newblock {\em Ricerche Mat.}, 7:102--137, 1958.

\bibitem{MR1404326}
Nicola Garofalo and Duy-Minh Nhieu.
\newblock Isoperimetric and {S}obolev inequalities for
  {C}arnot-{C}arath\'{e}odory spaces and the existence of minimal surfaces.
\newblock {\em Comm. Pure Appl. Math.}, 49(10):1081--1144, 1996.

\bibitem{MR3243741}
Loukas Grafakos.
\newblock {\em Modern {F}ourier analysis}, volume 250 of {\em Graduate Texts in
  Mathematics}.
\newblock Springer, New York, third edition, 2014.

\bibitem{MR2576685}
Philip~T. Gressman.
\newblock {$L^p$}-improving properties of averages on polynomial curves and
  related integral estimates.
\newblock {\em Math. Res. Lett.}, 16(6):971--989, 2009.

\bibitem{MR3300318}
Larry Guth.
\newblock A short proof of the multilinear {K}akeya inequality.
\newblock {\em Math. Proc. Cambridge Philos. Soc.}, 158(1):147--153, 2015.

\bibitem{2018arXiv181112559H}
Terence L.~J. {Harris}.
\newblock {An a.e. lower bound for Hausdorff dimension under vertical
  projections in the Heisenberg group}.
\newblock {\em Ann. Acad. Sci. Fenn. Math. (to appear)}, 2020.

\bibitem{2020arXiv200204789H}
Terence L.~J. {Harris}, Chi N.~Y. {Huynh}, and Fernando {Roman-Garcia}.
\newblock {Dimension Distortion by Right Coset Projections in the Heisenberg
  Group}.
\newblock {\em arXiv e-prints, to appear in Journal of Fractal Geometry}, page
  arXiv:2002.04789, February 2020.

\bibitem{MR0358443}
Walter Littman.
\newblock {$L^{p}-L^{q}$}-estimates for singular integral operators arising
  from hyperbolic equations.
\newblock In {\em Partial differential equations ({P}roc. {S}ympos. {P}ure
  {M}ath., {V}ol. {XXIII}, {U}niv. {C}alifornia, {B}erkeley, {C}alif., 1971)},
  pages 479--481, 1973.

\bibitem{MR0031538}
L.~H. Loomis and H.~Whitney.
\newblock An inequality related to the isoperimetric inequality.
\newblock {\em Bull. Amer. Math. Soc}, 55:961--962, 1949.

\bibitem{MR667786}
D.~M. Oberlin and E.~M. Stein.
\newblock Mapping properties of the {R}adon transform.
\newblock {\em Indiana Univ. Math. J.}, 31(5):641--650, 1982.

\bibitem{MR1887100}
Daniel~M. Oberlin.
\newblock Convolution with measures on polynomial curves.
\newblock {\em Math. Scand.}, 90(1):126--138, 2002.

\bibitem{MR676380}
Pierre Pansu.
\newblock Une in\'{e}galit\'{e} isop\'{e}rim\'{e}trique sur le groupe de
  {H}eisenberg.
\newblock {\em C. R. Acad. Sci. Paris S\'{e}r. I Math.}, 295(2):127--130, 1982.

\bibitem{MR3587666}
F.~Serra~Cassano.
\newblock Some topics of geometric measure theory in {C}arnot groups.
\newblock In {\em Geometry, analysis and dynamics on sub-{R}iemannian
  manifolds. {V}ol. 1}, EMS Ser. Lect. Math., pages 1--121. Eur. Math. Soc.,
  Z\"urich, 2016.

\bibitem{MR2895344}
Betsy Stovall.
\newblock {$L^p$} improving multilinear {R}adon-like transforms.
\newblock {\em Rev. Mat. Iberoam.}, 27(3):1059--1085, 2011.

\bibitem{MR1969206}
Terence Tao and James Wright.
\newblock {$L^p$} improving bounds for averages along curves.
\newblock {\em J. Amer. Math. Soc.}, 16(3):605--638, 2003.

\end{thebibliography}

\end{document}